\documentclass[a4paper]{amsart}
\usepackage[utf8]{inputenc}
\usepackage[T1]{fontenc}
\usepackage{lmodern, enumerate}
\usepackage{amssymb,amsxtra}
\usepackage{amscd}
\usepackage[all]{xy}
\usepackage{xcolor}
\usepackage{nicefrac,mathtools}
\usepackage{microtype}
\usepackage{tikz-cd}
\usepackage{tkz-graph}
\usetikzlibrary{arrows}
\usepackage{comment}
\usetikzlibrary{positioning}
\usepackage[pdftitle={Free actions on separated graph C*-algebras},
 pdfauthor={Alcides Buss, Pere Ara, Ado Dalla Costa},
 pdfsubject={Mathematics}
]{hyperref}
\usepackage[lite]{amsrefs}
\renewcommand*{\MR}[1]{ \href{http://www.ams.org/mathscinet-getitem?mr=#1}{MR #1}}
\newcommand*{\arxiv}[1]{\href{http://www.arxiv.org/abs/#1}{arXiv: #1}}

\numberwithin{equation}{section}
\theoremstyle{plain}
\newtheorem{theorem}[equation]{Theorem}
\newtheorem{lemma}[equation]{Lemma}
\newtheorem{proposition}[equation]{Proposition}
\newtheorem{corollary}[equation]{Corollary}
\theoremstyle{definition}
\newtheorem{definition}[equation]{Definition}
\theoremstyle{remark}
\newtheorem{remark}[equation]{Remark}
\newtheorem{example}[equation]{Example}

\DeclareMathOperator{\cspn}{\overline{span}}

\DeclareMathOperator{\id}{\mathrm{id}}

\newcommand*{\K}{\mathcal K} 
\newcommand*{\Hi}{\mathcal H} 
\newcommand*{\A}{\mathcal A}

\newcommand*{\J}{\mathcal J}
\newcommand*{\M}{\mathcal M} 

\newcommand*{\C}{\mathbb C}
\newcommand*{\Z}{\mathbb Z}

\newcommand*{\F}{\mathbb F}

\newcommand*{\f}{\mathfrak f}

\newcommand*{\nb}{\nobreakdash}
\newcommand*{\Star}{\(^*\)\nobreakdash-}

\newcommand{\Bound}{\mathbb B}
\newcommand{\Mat}{\mathbb M}

\newcommand*{\cont}{C}
\newcommand*{\contz}{\cont_0}

\newcommand*{\Id}{\textup{id}}

\newcommand*{\CG}{\mbox{$C^*(G)$}} 

\newcommand*{\CE}{\mbox{$C^*(E)$}} 

\newcommand*{\CEC}{\mbox{$C^*(E,C)$}} 
\newcommand*{\CEGC}{\mbox{$C^*(E \times_c G, C \times_c G)$}}  

\newcommand*{\CECR}{\mbox{$C_r^*(E,C)$}} 
\newcommand*{\CEGCR}{\mbox{$C_r^*(E \times_c G, C \times_c G)$}}  

\newcommand*{\defeq}{\mathrel{\vcentcolon=}}
\newcommand*{\congto}{\xrightarrow\sim}

\newcommand*{\sbe}{\subseteq} 
\newcommand*{\Free}{\mathbb F}
\newcommand*{\Cuntz}{\mathcal O}
\newcommand*{\cstar}{\texorpdfstring{$C^*$\nobreakdash-\hspace{0pt}}{*-}}
\newcommand*{\into}{\hookrightarrow}
\newcommand*{\onto}{\twoheadrightarrow}
\newcommand*{\red}{r}
\renewcommand*{\max}{\mathrm{max}}

\newcommand*{\dual}[1]{\widehat{#1}}

\newcommand{\spn}{\text{span}}
\newcommand{\spnfecho}{\overline{\text{span}}}
\newcommand{\Aut}{\text{Aut}}

\newcommand{\ol}{\overline}

\renewcommand{\P}{P}
\newcommand{\NN}{\mathcal{N}}
\newcommand{\LL}{\mathcal{L}}

\newcommand{\freeprod}[2]{\underset{#1}\bigstar{#2}}
\newcommand{\rfreeprod}[2]{\underset{#1,\red}\bigstar{#2}}

\newcommand{\idealin}{\mathrel{\triangleleft}} 

\newcommand{\fl}{\mathfrak f}

\begin{document}
\title[Free actions of groups on separated graph \cstar{}algebras]{Free actions of groups\\ on separated graph \cstar{}algebras}
\author{Pere Ara}  
\address[P. Ara]{Departament de Matem\`atiques, Edifici Cc, Universitat Aut\`onoma de Barcelona, 08193 Cerdanyola del Vall\`es (Barcelona), Spain, and}
\address{Centre de Recerca Matem\`atica, Edifici Cc, Campus de Bellaterra, 08193 Cerdanyola del Vall\`es (Barcelona), Spain.}
\email{para@mat.uab.cat}
\author{Alcides Buss}  
 \address[A. Buss]{Departamento de Matem\'atica\\
  Universidade Federal de Santa Catarina\\
  88.040-900 Florian\'opolis-SC\\
  Brazil}
\email{alcides.buss@ufsc.br}
\author{Ado Dalla Costa} 
\address[A.D. Costa]{Departamento de Matem\'atica\\
	Universidade Federal de Santa Catarina\\
	88.040-900 Florian\'opolis-SC\\
	Brazil}
\email{adodallacosta@hotmail.com}

\begin{abstract}
In this paper we study free actions of groups on separated graphs and their \cstar{}algebras, generalizing previous results involving ordinary (directed) graphs.

We prove a version of the Gross-Tucker Theorem for separated graphs yielding a characterization of free actions on separated graphs via a skew product of the (orbit) separated graph by a group labeling function. Moreover, we describe the \cstar{}algebras associated to these skew products as crossed products by certain coactions coming from the labeling function on the graph. Our results deal with both the full and the reduced \cstar{}algebras of separated graphs.

To prove our main results we use several techniques that involve certain canonical conditional expectations defined on the \cstar{}algebras of separated graphs and their structure as amalgamated free products of ordinary graph \cstar{}algebras. Moreover, we describe Fell bundles associated with the coactions of the appearing labeling functions. As a byproduct of our results, we deduce that the \cstar{}algebras of separated graphs always have a canonical Fell bundle structure over the free group on their edges.
\end{abstract}


\thanks{This work has been supported by CNPq/Humboldt-CAPES--Brazil. The first named author was partially supported by DGI-MINECO-FEDER grant PID2020-113047GB-I00, and the Spanish State Research Agency, through the Severo Ochoa and Mar\'ia de Maeztu Program for Centers and Units of Excellence in R$\&$D (CEX2020-001084-M)}


\maketitle

\tableofcontents

\section{Introduction}\label{cap0}
Graph \cstar{}algebras provide an important class of \cstar{}algebras that have been used both as a tool to study large classes of \cstar{}algebras supplying models for the classification theory and as a source of an inexhaustible supply of examples and counterexamples for many other areas of mathematics. The first appearance of this type of algebra was developed by Leavitt in \cite{Leavitt:Module_type}, who gives us a construction of a class of algebras denoted by $L_{\mathbb{K}}(m,n)$ over an arbitrary field $\mathbb{K}$ for integers $1\leq m \leq n$. The algebras $L_{\mathbb{K}}(m,n)$ are universal $(m,n)$-paradoxical algebras, in the sense that there exists a universal isomorphism between the free module of rank $m$ and the free module of rank $n$. In 1979, Cuntz independently constructed the Cuntz algebra $\mathcal{O}_n$  \cite{Cuntz:Simple_isometries}, which is the universal \cstar{}algebra generated by $n$ isometries $S_1,\ldots,S_n$ such that $S_i^*S_j = \delta_{i,j}$ and $\sum_{i=1}^n S_iS_i^* = 1$. It is the most basic example of a graph \cstar{}algebra and it is the subject of many studies in several areas. We point out that no explicit connections with graph theory were developed in these pioneering articles. Some years later, in \cite{Cuntz-Krieger:topological_Markov_chains}, Cuntz and Krieger generalized the Cuntz algebras to accomplish classes of \cstar{}algebras associated to more general Cuntz-type relations involving finite square matrices with entries in $\{0,1\}$. Subsequently, the \cstar{}algebra $C^*(E)$ associated with a directed graph $E$ was defined, and it was realized that these graph \cstar{}algebras provide direct generalizations of Cuntz algebras and Cuntz-Krieger algebras, see \cite{Kumjian-Pask-Raeburn-Renault:Graphs} and  \cite{Kumjian-Pask-Raeburn:Cuntz-Krieger_graphs}.  

Years later a generalization of graph \cstar{}algebras was presented by the first-named author and K.R. Goodearl (see \cite{Ara-Goodearl:C-algebras_separated_graphs} and \cite{Ara-Goodearl:Leavitt_path}). These  \cstar{}algebras are based on the concept of a separated graph, which consists of a pair $(E,C)$ where $E$ is a directed graph and $C$ is a partition of the set of edges of $E$ which refines the partition determined by the source function $s\colon E^1\to E^0$. If one takes as $C$ the partition of $E^1$ determined by the sets $ s^{-1}(v)$, where $v$ is a non-sink vertex of $E$, the \cstar{}algebra $\CEC$ coincides with $C^*(E)$ as expected. This concept of separated graph \cstar{}algebras is also related to the \cstar{}algebras of edge-colored graphs introduced by Duncan in \cite{Duncan:Certain_free}.

The main motivations to study these general classes of \cstar{}algebras come from the fact that they allow for a more complicated ideal structure as well as more interesting K\nb-theory groups than those associated to ordinary graphs. Moreover, the \cstar{}algebras $\CEC$ of separated graphs are usually not nuclear and have a more ``wild'' behaviour than $\CE$. Indeed, $\CEC$ can be viewed as an amalgamated free product of \cstar{}algebras $C^*(E_X)$ for the subgraphs $E_X\sbe E$ coming from different sets $X$ in the partition $C$. One can then also define the reduced versions $\CECR$ as the corresponding reduced amalgamated free products with respect to certain canonical conditional expectations $C^*(E_X)\onto \contz(E^0)$. These reduced $C^*$-algebras $\CECR$ have been defined so far only for {\it finitely separated} graphs $(E,C)$. All this makes the study of separated graph \cstar{}algebras much more interesting, and a lot of problems are still open. Furthermore, the amalgamated free product structure of these algebras makes it clear that, in general, a product of partial isometries coming from the different pieces $E_X$ is not a partial isometry anymore. In addition, separated graph algebras are a key ingredient in the solution of important problems outside the theory of graph algebras. For instance, they have been used in \cite{Ara-Exel:Dynamical_systems} to show that Tarski's dichotomy (\cite[Section 11.1]{TW}) does not hold in topological dynamics, and in \cite{ABP2020} to provide a solution to the Realization Problem for finitely generated refinement monoids.       

The main goal of this paper is to study free actions of groups on separated graphs and on their associated \cstar{}algebras.  
The starting point is the Gross-Tucker Theorem characterizing graphs carrying free actions of groups, proved for ordinary graphs by Gross-Tucker in \cite[Theorem 2.2.2]{Gross-Tucker:topolgical-graph-theory}. Our first main result extends this to separated graphs:
\begin{theorem}[Gross-Tucker for separated graphs]\label{GT-intro}
Let $(F,D)$ be a separated graph endowed with a free action of $G$. Then $(F,D)$ is $G$-equivariantly isomorphic to a skew product separated graph of the form
$$(F,D)\cong (E\times_c G,C\times_c G)$$
for some (labeling) function $c\colon E^1\to G$. Here $G$ acts on $E\times_c G$ by acting trivially on $E$ and by translations on $G$.
\end{theorem} 
The graph $E$ can be canonically identified with the quotient graph $F/G$ by taking orbits of the group action and similarly $C\cong D/G$. By this result, the problem of studying free actions of groups on separated graphs turns into the equivalent problem of studying skew product separated graphs $(E\times_c G, C\times_c G)$. 

The first basic question one can address in this context is to determine the exact relationship between $C^*(E\times_c G,C\times_c G)$ and $C^*(E,C)$. 
This was first answered by Kumjian-Pask \cite{Kumjian-Pask:C-algebras_directed_graphs} in case of ordinary graphs and for an abelian group $G$, in which case it is proved that the labeling $c\colon E^1\to G$ induces a (strongly continuous) action $\beta_c$ of the dual (compact) group $\dual G$ on $C^*(E)$ in such a way that
$C^*(E\times_c G)\cong C^*(E)\rtimes_{\beta_c} \dual G.$
Shortly thereafter, in \cite[Theorem 2.4]{Kaliszewski-Quigg-Raeburn:Skew_products}, the commutativity  assumption on $G$ was removed, replacing the action $\beta_c$ of $\dual G$ by a coaction $\delta_c$ of $G$ on $C^*(E)$, yielding a generalized form of the above isomorphism in terms of crossed products by coactions:
$C^*(E\times_c G)\cong C^*(E)\rtimes_{\delta_c}G$.
We extend this result to separated graphs as follows:
\begin{theorem}\label{theo:coactions-separated-intro}
Given a separated graph $(E,C)$ and a labeling function $c\colon E^1 \to G$, there is a coaction $\delta_c$ of $G$ on $C^*(E,C)$ acting on the partial isometry generators $S_e\in C^*(E,C)$ by $\delta_c(S_e)=S_e\otimes c(e)$ and such that there is a canonical isomorphism
$$C^*(E,C)\rtimes_{\delta_c}G\cong C^*(E\times_c G,C\times_c G).$$
Moreover, if in addition $(E,C)$ is a finitely separated graph, then $\delta_c$ factors through a coaction $\delta_c^r$ of $G$ on $C^*_r(E,C)$ in such a way that
$$C^*_r(E,C)\rtimes_{\delta_c^r}G\cong C^*_r(E\times_c G,C\times_c G).$$
\end{theorem}

Moreover, we prove that the coaction $\delta_c$ on $C^*(E,C)$ is \emph{maximal}\footnote{See Section~\ref{sec:Preliminaries} below for the terminology used here.} and that its dual action on $C^*(E,C)\rtimes_{\delta_c}G$ corresponds under the above isomorphism to the action $\alpha$ on $C^*(E\times_c G,C\times_c G)$ induced by the free action on $(E\times_c G,C\times_c G)$ appearing in the Gross-Tucker Theorem. Similarly, the coaction $\delta_c^r$ on $C^*_r(E,C)$ is \emph{normal} and its dual coaction on $C^*_r(E,C)\rtimes_{\delta_c^r}G$ corresponds to a $G$-action $\alpha_r$ on $C^*_r(E\times_c G,C\times_c G)$ which factors the $G$-action $\alpha$ on $C^*(E\times_c G,C\times_c G)$.

Using the Gross-Tucker Theorem~\ref{GT-intro} and duality for crossed products by coactions, Theorem~\ref{theo:coactions-separated-intro} yields a `dual' result that describes crossed products by free actions on separated graph \cstar{}algebras:

\begin{theorem}\label{theo:coactions-separated-intro-1}
For a free action of a group $G$ on a separated graph $(F,D)$, there is an induced $G$-action $\alpha$ on $C^*(F,D)$ and a canonical isomorphism
$$C^*(F,D)\rtimes_{\alpha}G\cong C^*(F/G,D/G)\otimes\K(\ell^2G).$$
Moreover, if $(F,D)$ is finitely separated, then $\alpha$ factors through a $G$-action $\alpha_r$ on $C^*_r(F,D)$ and the above isomorphism factors through
$$C^*_r(F,D)\rtimes_{\alpha_r,r}G\cong C^*_r(F/G,D/G)\otimes\K(\ell^2G).$$
\end{theorem}

The case of separated graphs is much more interesting and intriguing because it brings into play non-amenable phenomena that do not appear in ordinary graphs. The main point in the ordinary graph case is that the \cstar{}algebra $C^*(E)$ attached to a graph $E$ is always nuclear. This implies that all coactions of a discrete group on such a \cstar{}algebra are both maximal and normal. Furthermore, a free action of $G$ on a graph $F$ always induces an amenable action on $C^*(F)$ in the sense of Anantharaman-Delaroche \cite{Anantharaman-Delaroche:Systemes} because the crossed product $C^*(F)\rtimes_\alpha G\cong C^*(F/G)\otimes \K(\ell^2G)$ is Morita equivalent to a graph \cstar{}algebra, hence also nuclear.

All this no longer happens for separated graphs: the maximal coaction $\delta_c$ on $C^*(E,C)$ induced from a labeling function is not necessarily normal, 
and the normal coaction $\delta_c^r$ on $C^*_r(E,C)$ is not necessarily maximal. Even more interesting, the coaction $\delta_c^r$ does not need to be the normalization of $\delta_c$. This also means that the actions $\alpha$ on $C^*(F,D)$ or its reduced companion $\alpha_r$ on $C^*_r(F,D)$ induced by a free action of $G$
on $(F,D)$ need not be amenable in general.

In order to better understand the relation between the coaction $\delta_c$ on $C^*(E,C)$ and its quotient coaction $\delta_c^r$ on $C^*_r(E,C)$, we describe their spectral decompositions, that is, the associated $G$-grading given by the subspaces
$$A_g=C^*(E,C)_g:=\{a\in C^*(E,C):\delta_c(a)=a\otimes g\}$$
for a fixed $g\in G$. We prove that $A_g$ is the closed linear span of elements of the form
$$S_{\mu_1}S_{\nu_1}^* \ldots S_{\mu_n}S_{\nu_n}^*$$ 
for certain reduced paths of the form $\varsigma=\mu_1\nu_1^*\ldots \mu_n\nu_n^*$ with $c(\varsigma)=g$, where $c$ is extended to paths in the canonical way. When endowed with the algebraic structure of $C^*(E,C)$, this yields a Fell bundle $\A=\{A_g\}_{g\in G}$ over $G$. And since the coaction $\delta_c$ is maximal, we have a canonical isomorphism $C^*(E,C)\cong C^*(\A)$, the cross-sectional \cstar{}algebra of $\A$. Similarly, the coaction $\delta_c^r$ yields a Fell bundle $\A^r=\{A^r_g\}_{g\in G}$, where $A^r_g=C^*_r(E,C)_g$ are the spectral subspaces of $\delta_c^r$ which can be described from reduced paths as above. And since $\delta_c^r$ is normal we have $C^*_r(E,C)\cong C^*_r(\A^r)$. The fact that the coaction $\delta_c^r$ is, in general, not the normalization of $\delta_c$ is related to the fact that $\A^r$ is a proper quotient Fell bundle of $\A$. We describe this quotient in terms of the canonical conditional expectation $P\colon C^*(E,C)\to C_0(E^0)$ and its nucleus $\NN_P:=\{a\in C^*(E,C): P(a^*a)=0\}$, which gives a closed two-sided ideal of $C^*(E,C)$ realizing $C^*_r(E,C)\cong C^*(E,C)/\NN_P$. All this can also be described on the level of the Fell bundles $\A$ and $\A^r$ as follows:

\begin{theorem}\label{Fell-bundle-quotient-intro}
The nucleus ideal $\NN_P$ gives a Fell bundle ideal $\J=(A_g\cap \NN_P)_{g\in G}$ of $\A$ whose quotient Fell bundle $\A/\J$ is canonically isomorphic to $\A^r$.
In particular, $$C^*_r(E,C)\cong C^*_r(\A/\J).$$
\end{theorem}

From this we see that the normalization of $\delta_c$ -- which corresponds to the dual coaction on $C^*_r(\A)$ --  
can be realized on some `exotic' \cstar{}algebra lying between $C^*(E,C)$ and $C^*_r(E,C)$. Similarly, the maximalization of $\delta_c^r$ -- which corresponds to the dual coaction on $C^*(\A/\J)$ -- also takes place on a \cstar{}algebra lying between $C^*(E,C)$ and $C^*_r(E,C)$.
And it may happen that $\J$ is not zero and all these \cstar{}algebras are different. Indeed, $\J=0$ if and only if $A_1\cap \NN_P=0$ (where $A_1\cap \NN_P$ is the unit fiber of $\J$, i.e., the fiber associated to the unit $1\in G$), and this happens if and only if $P$ restricts to a faithful conditional expectation on $A_1$. 

Among all groups and  labeling functions $c\colon E^1\to G$ we can possibly consider on a graph $E$, there exists always a canonical label, the so-called \emph{free label} $\fl\colon E^1\to \Free$, where $\Free$ is the free group on the set $E^1$ and $\fl$ is then the canonical embedding. Applying our results to this label we get the following:

\begin{corollary}\label{Fell-bundle-quotient-intro-cor}
Every separated graph \cstar{}algebra $C^*(E,C)$ has a canonical $\Free$-grading structure yielding a Fell bundle $\A$ over $\Free$ with $C^*(\A)\cong C^*(E,C)$. And if $(E,C)$ is finitely separated, the quotient of $\A$ by the nucleus $\NN_P$ of the conditional expectation $P\colon C^*(E,C)\to C_0(E^0)$ yields a $\Free$-grading structure for $C^*_r(E,C)$ with an associated Fell bundle $\A^r$ isomorphic to the quotient $\A/\J$ by the ideal $\J=\{A_g\cap\NN_P\}_{g\in G}$ of $\A$, so that $C^*_r(\A/\J)\cong C^*_r(E,C)$.
\end{corollary} 

Since the free label only depends on the graph, this gives an intrinsic Fell bundle structure on every separated graph \cstar{}algebra. We describe these
Fell bundles in certain examples.

\medskip

This paper is structured in four main parts, as follows. Section \ref{sec:Preliminaries} is dedicated to reviewing the background and preliminaries needed in the following sections. Mainly to fix our conventions and notations, we introduce graphs and their skew product graphs with groups. We also review the theory of crossed products by coactions, and explore the connections between coactions of discrete groups and Fell bundles. 

In Section~\ref{cap2} we focus on full separated graph \cstar{}algebras $C^*(E,C)$, proving the Gross-Tucker Theorem~\ref{GT-intro} and parts of Theorems~\ref{theo:coactions-separated-intro} and~\ref{theo:coactions-separated-intro-1}, the ones dealing with full \cstar{}algebras of separated graphs. 

In Section~\ref{cap3}, we review the construction from  \cite{Ara-Goodearl:C-algebras_separated_graphs} of reduced separated graph \cstar{}algebras $C^*_r(E,C)$ through the theory of reduced amalgamated free products that goes back to Voiculescu \cite{Voiculescu:Symmetries}. Although this is an elegant definition, it is not very easy to handle in practice. We propose to look at these \cstar{}algebras from the point of view of the canonical conditional expectation
$P\colon C^*(E,C)\to C_0(E^0)$ and recover $C^*_r(E,C)$ from the quotient $C^*(E,C)/\NN_P$ by the nucleus $\NN_P$ of $P$, which is the largest closed two-sided ideal contained in $\ker(P)$. Indeed, this can be defined for every conditional expectation $P\colon A\to B$ and gives a special ``reduced'' quotient $A_{P,r}$ carrying an (almost) faithful conditional expectation onto $B$, see \cite{Kwasniewski-Meyer:Essential}. This construction is naturally functorial and behaves well with many other constructions, like tensors and crossed products by coactions. This machinery will be then used to prove the statements involving the reduced separated graph \cstar{}algebras in Theorems~\ref{theo:coactions-separated-intro} and~\ref{theo:coactions-separated-intro-1}.

Finally, in Section~\ref{sec:Fell-bundle} we describe the Fell bundles associated to the coactions on separated graph \cstar{}algebras coming from labeling functions via their spectral decompositions. In particular we prove Theorem~\ref{Fell-bundle-quotient-intro} and Corollary~\ref{Fell-bundle-quotient-intro-cor}. We also present some concrete examples of separated graphs and describe their free label Fell bundle structures.

\section{Preliminaries}\label{sec:Preliminaries}

In this section, we summarize some of the main tools we are going to use throughout the paper. This will also consolidate the notation we use.

\subsection{Directed graphs}\label{sec:DirectedGraphs}
	A \emph{directed graph} $E$ is a quadruple of the form $E=(E^0,E^1,s,r)$ consisting of two sets $E^0$,$E^1$ and two maps $s,r: E^1 \rightarrow E^0$. The elements of $E^0$ and $E^1$ are called vertices and edges and the maps $s,r$ are called the source and range maps. 
		
	The graph is called \emph{row-finite} if every vertex emits at most finitely many edges. We follow the convention of composing paths, that is, a \emph{finite path} in $E$ is a sequence of edges of the form $\mu:= e_1 \ldots e_n$ with $r(e_i)=s(e_{i+1})$ for all $i \in \{1, \ldots, n-1\}$. The length of $\mu$ is $|\mu|:=n$ and paths with length $0$ are identified with the vertices of $E$ (we set $s(v)=r(v)=v$). We denote by $E^n$ the set of all finite paths with length $n$ and $\text{Path}(E) := \displaystyle\cup_{n=0}^{\infty} E^n$ the set of all paths of $E$. We can extend the source and range maps to $\text{Path}(E)$ in the obvious way: if $\mu=e_1\ldots e_n \in \text{Path}(E)$, then $s(\mu) = s(e_1)$ and $r(\mu) = r(e_n)$. Given two paths $\mu,\nu \in \text{Path}(E)$ with $r(\mu) = s(\nu)$, one obtains a new path $\mu\nu$ by concatenation with $|\mu\nu| = |\mu|+|\nu|$. 
	
	Given a graph $E$, we define its {\it extended graph}  (also called the {\it double graph} of $E$) as the new graph $\hat{E}=(E^0,E^1\cup (E^1)^*, r,s)$, where $(E^1)^*= \{e^* : e\in E^1\}$ and we set $s(e^*)= r(e)$, $r(e^*)= s(e)$ for all $e\in E^1$. Moreover, given two graphs $E$ and $F$, we define a {\it graph morphism} $f: E \to F$ as a pair of maps $f=(f^0, f^1)$ where $f^i:E^i \to F^i$ $(i=0,1)$ which commute with source and range maps, that is, satisfy $f^0(r(e)) = r(f^1(e))$ and $f^0(s(e)) = s(f^1(e))$, for all $e \in E^1$. If $f^0$ and $f^1$ are bijective maps then $f$ is called a {\it graph isomorphism}. There is a natural notion of automorphism of graphs and the collection of all automorphisms of a graph $E$ forms a group under composition, denoted by $\Aut(E)$. 
	
	All groups in this paper are discrete. An \emph{action} of $G$ on a (directed) graph $E$ consists of actions on the set of vertices and edges of $E$ in such a way that the source and range maps are $G$-equivariant: $s(g\cdot e)=g\cdot s(e)$ and $r(g\cdot e)=g\cdot r(e)$.	 An action can also be viewed as a group homomorphism $\alpha: G \rightarrow \Aut(E)$. It is called \emph{free} if it acts freely on the vertices. If $G$ acts freely on the vertices then it also acts freely on the edges because of the existence of the equivariant maps $s,r\colon E^1\to E^0$.

\subsection{Skew products of graphs}
\label{subsection:skewproductsofgraphs}
 Let $E$ be a (directed) graph. Let us also fix a discrete group $G$ and a function $c: E^1 \rightarrow G$, which we interpret as a \emph{labeling function}.

\begin{definition}
	With notations as above, we define the \emph{skew product graph} $$E \times_c G := (E^0 \times G,E^1 \times G, s,r )$$ in which the sets of vertices and edges are the Cartesian products of $E^0$ and $E^1$ with $G$, respectively, and the source and range maps (denoted by the same letters) are defined by $$s(e,g) = (s(e),g) \quad\text{ and } \quad r(e,g) = (r(e),gc(e))$$ for all $e \in E^1$ and $g \in G$. 
\end{definition}

If $E$ is row-finite, then so is $E\times_c G$ since $s^{-1}(v,g) = s^{-1}(v) \times \{g\}$ by definition. In the literature $E \times_c G$ is referred to as the \emph{derived graph} or \emph{voltage graph}, see \cite{Gross-Tucker:topolgical-graph-theory}. Skew product graphs have many applications, for instance, they are used in the theory of branched coverings of surfaces, see \cite{Gross-Tucker:generating}. There are different types of conventions here: for us it is better to define the skew product graph $E \times_c G$ by using the above formula, since we are focusing on $s^{-1}(v)$ instead of $r^{-1}(v)$, and also to make the results of our main theorems later more natural. Our convention is compatible with \cite{Gross-Tucker:topolgical-graph-theory}, and different from those in \cite{Kumjian-Pask:C-algebras_directed_graphs}, \cite{Kaliszewski-Quigg-Raeburn:Skew_products}.

We may extend the labeling function $c$ to $\text{Path}(E)$ by defining $c(v) := 1$ for every $v \in E^0$ and $c(\mu) := c(e_1)\cdots c(e_n)$ for every $\mu = e_1 \cdots e_n \in \text{Path}(E)$. Observe that we have $c(\mu\nu) = c(\mu)c(\nu)$ for every $\mu, \nu \in \text{Path}(E)$ with $r(\mu) = s(\nu)$. 
We may consider $\text{Path}(E)$ as a (small) category, with set of objects $E^0$, called the {\it path category} of $E$. The group $G$ is a category with a single object, in which each morphism is invertible. We thus see that any labeling $c\colon E^1\to G$ induces a functor $c\colon \text{Path}(E)\to G$ from the path category of $E$ to $G$.

The labeling function $c$ can be extended to a labeling function on the extended graph $\hat{E}$, also denoted by $c$, by $c(e^*)=c(e)^{-1}$. This allows us to define $c$ on $\text{Path}(\hat{E})$, with the property that $c(\mu \nu)= c(\mu)c(\nu)$ whenever $\mu,\nu\in \text{Path}(\hat{E})$ and $r(\mu)=s(\nu)$. Hence every labeling $c\colon E^1\to G$ induces a $*$-functor $c\colon \text{Path}(\hat{E})\to G$ from the $*$-category $\text{Path}(\hat{E})$ to the $*$-category $G$ (where the involution in the latter is given by the inverse). 


For each $\mu \in \text{Path}(E)$ of the form $\mu=e_1\cdots e_n$, with $r(e_i)=s(e_{i+1})$ for $i \in \{1,\ldots,n-1\}$, and $g \in G$, we define a path $(\mu,g)$ in the skew product graph by: 
\begin{align}\label{caminhografoskew}
(\mu,g) := (e_1,g)(e_2,gc(e_1))\cdots (e_n,gc(e_1\cdots e_{n-1})).
\end{align}
 Due to the definition of range and source maps in the skew product graph, all paths in $E \times_c G$ are of this form. For example, if $e,f \in E^1$ and $g,h \in G$, the short path $(e,g)(f,h)$ makes sense if $r(e,g) = s(f,h)$, which means that $r(e)=s(f)$ and $h=gc(e)$. The notation $(\mu,g)$ will be used exclusively for the paths in $E\times_c G$ defined as in \eqref{caminhografoskew}. One should not confuse it with $(\mu,g) = (e_1\cdots e_n,g)$ which is not a path on $E \times_c G$.

Let $G$ act on a graph $E$. We define the \emph{quotient graph} $$E/G:=((E/G)^0,(E/G)^1,s_G,r_G)$$ where $(E/G)^0 = E^0/G$ and $(E/G)^1 = E^1/G$ are the equivalence classes of vertices and edges respectively under the action of $G$ and the source and range maps are defined to be $s_G([e]) = [s(e)]$ and $r_G([e]) = [r(e)]$. It is easy to check that the source and range maps are well defined because the action commutes with both maps. Moreover, the quotient map $q: E \rightarrow E/G$ is a surjective graph morphism. 

If $c\colon E^1 \to G$ is a labeling function, then $E \times_c G$ is a graph which carries a natural free action $\gamma$ of $G$ by left multiplication: $$\gamma_g(v,h):=(v,gh) \quad \text{ and } \quad  \gamma_g(e,h):=(e,gh) $$ for all $v \in E^0$, $e \in E^1$ and $g \in G$. Moreover, the quotient of the skew product graph by the action $\gamma$ recovers $E$. As already recalled in the introduction, the theorem of Gross-Tucker \cite{Gross-Tucker:topolgical-graph-theory} shows that the skew product graphs are, up to isomorphism, exactly the graphs carrying free actions of groups.

\subsection{Coactions and their crossed products}\label{sec:coactions}

Now, we will introduce coactions of discrete groups and their crossed products. Although we only need discrete groups, we remark that this also generalizes to locally compact groups, see \cite{Quigg:FullAndReducedCoactions}. We will assume that all representations here are nondegenerate and all tensor products $\otimes$ are minimal with identity maps denoted by $\id$ with indices like $\id_A$ when necessary. We mainly follow \cite{Echterhoff-Quigg:InducedCoactions}, \cite{Landstad:Duality}, \cite{Echterhoff-Kaliszewski-Quigg:Maximal_Coactions},  \cite{Quigg:FullAndReducedCoactions},  \cite{Nilsen:Full_crossed},  \cite{Katayama:Takesaki_Duality},  \cite{Landstad-Philips-Raeburn:Representations_Coactions},  \cite{Kaliszewski-Quigg-Raeburn:Skew_products} and \cite{Nilsen:DualityCrossedProducts}.

First we recall that the group \cstar{}algebra $C^*(G)$ carries a natural comultiplication $\delta_G: \CG \rightarrow \CG \otimes \CG $ such that $\delta_G(g) =  g \otimes g $, where $g$ here actually means its image in $\CG$ under the universal representation. That $\delta_G$ is a comultiplication means $(\delta_G \otimes \id_G) \circ \delta_G = (\id_G \otimes \delta_G) \circ \delta_G$.

\begin{definition}\label{definicaocoacao}
	A \emph{coaction} of a discrete group $G$ on a \cstar{}algebra $A$ is a *-homomorphism $\delta: A \rightarrow A \otimes C^*(G)$ 
	satisfying the \emph{coaction identity} $(\delta\otimes\id)\circ\delta=(\id\otimes\delta_G)\circ\delta$ 
	and such that  it is \emph{nondegenerate} in the sense that $\overline{\delta(A)(1 \otimes C^*(G))} = A \otimes \CG$.
\end{definition}

	Coactions of an abelian discrete group $G$ correspond to (strongly continuous) actions of the (compact) Pontryagin dual group $\widehat{G}$, see \cite{Echterhoff-Kaliszewski-Quigg-Raeburn:Categorical}*{Example A.23}.

In general, a coaction of $G$ on $A$ induces a (left Banach algebra) action of $C^*(G)^*\cong B(G)$, the Fourier-Stieltjes algebra $B(G)$ of $G$ via the formula $f\cdot a:=(\id\otimes f)(\delta(a))$. Recall that $B(G)$ can be identified with the algebra consisting of all bounded functions on $G$ that can be expressed as matrix coefficients of unitary representations of $G$. The Fourier algebra $A(G)$ is the (closed) *-subalgebra of $B(G)$ consisting of all matrix coefficients of the left regular representation $\lambda\colon G\to \Bound(\ell^2(G))$. Every element $f \in A(G)$ is of the form $f(g)=\langle \xi,\lambda_g\eta\rangle$ with $\xi,\eta\in \ell^2(G)$. In particular, the characteristic function $\chi_g$ of $\{g\}$ belongs to $A(G)$ since $\chi_g = \langle \delta_g,\lambda(.)\delta_1\rangle_{\ell^2(G)}$, where $(\delta_g)_{g\in G}\subseteq \ell^2(G)$ is the standard basis. Then $\lambda$ is determined by $\lambda_g(\delta_h) = \delta_{gh}$ for all $g,h \in G$. Notice that $\chi_1$ is the canonical trace on $\CG$. More details on the Fourier algebra can be found in \cite{Eymard:FourierAlgebra}.  

Coactions of discrete groups are strongly related to Fell bundles. We refer the reader to \cite{Exel:Partial_amenable_free} and \cite{Exel:Partial_dynamical} for the theory of Fell bundles and their \cstar{}algebras.  
For a Fell bundle $\A= \{A_g \}_{g\in G}$ over a discrete group $G$, one can construct two \cstar{}algebras, the full \cstar{}algebra $C^*(\A)$ and the reduced \cstar{}algebra $C^*_r(\A)$. These carry canonical coactions of $G$, so in particular they are topologically $G$-graded \cstar{}algebras. This is defined in such a way that the fibers
 $A_g$ of $\A$ can be realized as the grading subspaces of both $C^*(\A)$ or $C^*_r(\A)$. And regarding $A_g$ as subspaces of these \cstar{}algebras, the coactions   
$\delta_\A: C^*(\A) \rightarrow C^*(\A) \otimes C^*(G)$ and $\delta_{\A}^r: C_r^*(\A) \rightarrow C_r^*(\A) \otimes C^*(G)$
act by $a_g\mapsto a_g\otimes g$ for $a_g\in A_g$.  These are generally called `dual coactions'. They generalize the dual coactions on crossed products when the underlying Fell bundle is a semi-direct product bundle $\A=A\times G$ associated with an action of $G$ on a \cstar{}algebra. More generally, one can consider even (twisted) partial actions here, see \cite{Exel:TwistedPartialActions}.

Conversely, if $\delta$ is a coaction of $G$ on a \cstar{}algebra $A$ we can consider the spectral subspaces 
$$A_g:=\{ a \in A ~|~ \delta(a)=a \otimes g \}.$$ 
This gives us a grading for $A$ over $G$ and consequently a Fell bundle $\A = \{A_g\}_{g \in G}$ over $G$ with operations induced from $A$. Moreover, this is a topological grading in the sense of \cite{Exel:Partial_dynamical}*{Definition~19.2}. This means that the grading admits pairwise orthogonal projections $E_g\colon A\to A_g$. Indeed, in the setting of coactions they are given by $E_g=(\Id\otimes\chi_g)\circ \delta$, that is, $E_g(a)=\chi_g\cdot a$, where $\chi_g$ is the characteristic function of $\{g\}$ regarded as an element of $A(G)\subseteq B(G)\cong C^*(G)^*$.

The correspondence between coactions and Fell bundles is described in more detail in \cite{Ng:Discrete-Coactions,Quigg:Discrete_coactions_and_bundles}.
This correspondence can be upgraded to certain equivalences if we range the coaction we deal with. More precisely, the assignments $\A\mapsto (C^*(\A),\delta_\A)$ and $\A\mapsto (C^*_r(\A),\delta_\A^r)$ are equivalences of categories between the category of Fell bundles over $G$ and \emph{maximal} or \emph{normal} coactions of $G$, respectively. We are going to explain this in more detail in what follows.



\begin{definition}\label{condicaodecovariancia}
	Let $(A,G,\delta)$ be a coaction and $B$ a \cstar{}algebra. A \emph{covariant representation} of $(A,G,\delta)$ in a multiplier 
\cstar{}algebra $\M(B)$ is a pair $(\pi,\mu)$ where $\pi: A \rightarrow \M(B)$ and $\mu: C_0(G) \rightarrow \M(B)$ are nondegenerate *-homomorphisms such that $$\pi(a) \mu(\chi_h) =  \mu(\chi_{gh})\pi(a) \text{ for all } a \in A_g, g,h \in G.$$  
\end{definition}

The definition above is special for discrete groups, see \cite[Lemma 3.1]{Deicke-Pask-Raeburn:Coverings} for the comparison
with the more general definition that applies to a locally compact groups. 

If $(\pi,\mu)$ is a covariant representation of $(A,G,\delta)$ into  $\M(B)$, then 
$C^*(\pi,\mu) := \cspn{\pi(A)\mu(C_0(G))}$ is a \cstar{}subalgebra of $\M(B)$ as can be shown directly from the covariance condition.
The so-called regular covariant representations of coactions are constructed in \cite{Raeburn:OnCrossedProductsByCoactions}*{Proposition 2.6}:
If $\pi$ is a nondegenerate *-homomorphism of $A$ to $\M(B)$ then the pair $((\pi \otimes \lambda)\circ \delta,1 \otimes M)$ is a covariant representation of $(A,G,\delta)$ into $\M(B \otimes \K(\ell^2 G))$, called the regular covariant representation induced by $\pi$. Applying this to the identity map $\id_A: A \to A$, we get the regular covariant representation $(j_A,j_G)=(\id_A \otimes \lambda)\circ \delta,1 \otimes M)$. The crossed product by the coaction is then defined as:
$$A \rtimes_\delta G := C^*(j_A,j_G)\sbe \M(A \otimes \K(\ell^2 G)).$$
This is, therefore, a concrete \cstar{}algebra, but it also has a universal property for covariant representations, see \cite[Theorem 4.1(b)]{Raeburn:OnCrossedProductsByCoactions}. More precisely
if $(\pi,\mu)$ is a covariant representation of $(A,G,\delta)$ into $\M(B)$, there is a unique nondegenerate *-homomorphism $\pi \times \mu: A \rtimes_\delta G\rightarrow \M(B) $ such that $\pi \times \mu \circ j_A = \pi$ and $\pi \times \mu \circ j_G = \mu$. The *-homomorphism $\pi \times \mu$ is called the integrated form of the covariant representation $(\pi,\mu)$ of $(A,G,\delta)$.

The coactions of $G$ form a category in a natural way: if $(A,G,\delta)$ and $(B,G,\epsilon)$ are two coactions, we say that a *-homomorphism $\varphi: A \rightarrow B$ is $G$-equivariant if 
$$ (\varphi \otimes \id_G) \circ \delta = \epsilon \circ \varphi.$$ 
The $G$-equivariance of $\varphi: A \to B$ can be verified in terms of the corresponding $G$-grading structures: $\varphi$ is $G$-equivariant if and only if it is a graded map in the sense that $\varphi(A_g) \subseteq B_g$ for all $g \in G$. The crossed product construction $(A,\delta)\mapsto A\rtimes_\delta G$ is a functor:
If $\varphi: A \rightarrow B$ is a $G$-equivariant *-homomorphism, then there is a unique induced *-homomorphism $\varphi \rtimes G: A \rtimes_\delta G \rightarrow B \rtimes_\epsilon G$ such that $\varphi \rtimes G(j_A(a)j_G^A(f)) = (j_B \circ \varphi)(a)j_G^B(f)$ for $a \in A$ and $f \in C_0(G)$, see \cite{Echterhoff-Kaliszewski-Quigg-Raeburn:Categorical}*{Lemma A.46}

As shown in \cite[Theorem 19.1]{Exel:Partial_dynamical} a topologically $G$-graded \cstar{}algebra $A$ with grading viewed as a Fell bundle $\A=\{A_g\}_{g \in G}$ always `lies' between $C^*(\A)$ and $C^*_r(\A)$ in the sense that the identity map on $\A$ induces a commutative diagram of surjective *-homomorphisms
	\begin{equation}\label{eq:diag-canonical-epi}
	\begin{tikzcd}
	C^*(\A)  \arrow{rr}{\Lambda } \arrow[swap]{dr}{\sigma} & &C_r^*(\A) \\
	& A \arrow{ur}{\psi} &
	\end{tikzcd}
	\end{equation}
where $\Lambda: C^*(\A) \to C_r^*(\A)$ denotes the regular representation of $\A$. If the topological $G$-grading is associated with a $G$-coaction $\delta\colon A\to A\otimes C^*(G)$, then we can say more. In this case it is straightforward to check that all the homomorphisms above are $G$-equivariant (with respect to the coactions $\delta_\A$, $\delta$ and $\delta_\A^r$), so they induce a commutative diagram of surjective *-homomorphisms:
	\begin{equation}\label{eq:diag-canonical-epi-integrated}
	\begin{tikzcd}
	C^*(\A) \rtimes_{\delta_\A} G \arrow{rr}{\Lambda \rtimes G} \arrow[swap]{dr}{\sigma \rtimes G} & &C_r^*(\A)\rtimes_{\delta_\A^r} G  \\
	& A\rtimes_{\delta} G \arrow{ur}{\psi \rtimes G} &
	\end{tikzcd}
	\end{equation}
Indeed, one can show that all maps in the above diagram~\eqref{eq:diag-canonical-epi-integrated} are isomorphisms, see \cite{Echterhoff-Quigg:InducedCoactions}*{Lemma 2.1}, although none of the homomorphisms in~\eqref{eq:diag-canonical-epi} are isomorphisms in general. This is exactly where the maximality or the normality of the coaction $\delta$ comes into play. Maximal and normal coactions are defined and studied in \cite{Exel:Partial_dynamical}, \cite{Echterhoff-Kaliszewski-Quigg:Maximal_Coactions} and \cite{Quigg:FullAndReducedCoactions}. We review the main ideas in what follows and complete the picture above.

For a coaction $\delta\colon A\to A\otimes C^*(G)$ we consider the canonical realization of the crossed product $A \rtimes_\delta G$ 
as a \cstar{}subalgebra of $\M(A \otimes \K(\ell^2 G))$ via the integrated form $\Pi:=j_A\rtimes j_G$ of the regular covariant pair $(j_A,j_G)$. The \emph{dual action} $\hat\delta$ of $G$ on $A\rtimes_\delta G$ can then be implemented via the unitary representation  $U:=\id_A \otimes \rho^G: G \to U\M(A \otimes \K(\ell^2 G))$ defined by $U_g := \id_A \otimes \rho_g^G$, where $\rho^G$ denotes the right regular representation of $G$. In other words, 
$$\hat\delta_g(x)=U_g x U_g^*\quad\forall g\in G.$$
This means that $(\Pi,U)$ is a covariant representation of the $G$-action $(A \rtimes_{\delta} G, G, \hat{\delta})$. 
And by the universal property of crossed products by actions we get a surjective *\nb-homomorphism 
\begin{align}\label{PI}
\Pi \times U: A \rtimes_\delta G \rtimes_{\widehat{\delta}} G \twoheadrightarrow A \otimes \K(\ell^2 G).
\end{align}
This is surjective because $\spnfecho\{M_f\rho_g ~|~ f \in C_0(G), g \in G\}$ is equal to $\K(\ell^2 G)$ and because $\delta$ is nondegenerate. The interesting fact is that this map is not always injective. 

By definition, the coaction $\delta: A \to A \otimes \CG$ is \emph{maximal} if $\Pi \times U$ is an isomorphism. And it is called \emph{normal} if $\Pi\rtimes U$ factors through an isomorphism
$$A \rtimes_\delta G \rtimes_{\widehat{\delta},r} G \congto A \otimes \K(\ell^2 G).$$

Here are then the main results connecting these concepts to Fell bundles (see \cite[Proposition 4.2]{Echterhoff-Kaliszewski-Quigg:Maximal_Coactions}):

\begin{proposition} 
	\label{propmaximalcoaçãofibras}
	Let $(A,G,\delta)$ be a coaction. Then $\delta:A \to A \otimes \CG$ is a maximal coaction if and only if the canonical map $\sigma: C^*(\A) \twoheadrightarrow A$ from~\eqref{eq:diag-canonical-epi}  
	is an isomorphism. And $\delta$ is normal if and only if the canonical surjection $\psi\colon A\to C^*_r(\A)$ from~\eqref{eq:diag-canonical-epi} is an isomorphism.
\end{proposition}

There are other characterizations of maximal or normal coactions. For instance, $\delta\colon A\to A\otimes C^*(G)$ is maximal if and only if it lifts to a (necessarily injective) homomorphism from $A$ to the maximal tensor product $A\otimes_{\max}C^*(G)$, see \cite{Buss-Echterhoff:Maximality}*{Theorem~5.1}. And $\delta$ is normal if only only if it factors through an injective homomorphism $\delta^r:=(\Id\otimes\lambda)\circ \delta\colon A\to A\otimes C^*_r(G)$, see \cite{Quigg:FullAndReducedCoactions}.
This, in turn, means that $\delta$ is \emph{normal} if and only if $j_A: A \to \M(A \rtimes_{\delta} G)$ is an injective map, or that the conditional expectation $E=(\Id\otimes\chi_1)\circ \delta\colon A\to A_1$ is faithful.

Every coaction admits a \emph{maximalization} which is unique up to isomorphism (see \cite[Theorem 3.3]{Echterhoff-Kaliszewski-Quigg:Maximal_Coactions}): a maximalization of $(A,\delta)$ is another coaction $(B,\epsilon)$ of $G$ endowed with a surjective $G$-equivariant homomorphism  $\varphi: B \to A$ such that $\varphi \rtimes G: B \rtimes_{\epsilon} G \to A \rtimes_{\delta} G$ is an isomorphism. Similarly, a \emph{normalization} of $(A,\delta)$ is a $G$-coaction $(C,\gamma)$ with a surjective homomorphism $\psi\colon A\onto C$ that induces an isomorphism $\psi\rtimes G\colon A\rtimes_\delta G\congto C\rtimes_\gamma G$. This holds even in the more general context of locally compact groups.
For discrete groups, this boils down again to looking at the associated Fell bundles: the diagram of isomorphisms~\eqref{eq:diag-canonical-epi-integrated} means that $(C^*(\A),\delta_\A)$ is the maximalization of $(A,\delta)$ and $(C^*_r(\A),\delta_\A^r)$ is its normalization if $\A$ is the Fell bundle associated with $\delta$.

\section{Separated Graph $C^*$-algebras}\label{cap2}
In this section, our aim is to obtain a version of the Gross-Tucker Theorem, allowing us to characterize exactly the separated graphs that carry a free group action, and to obtain some duality results involving separated graph \cstar{}algebras, generalizing previous works on ordinary graph \cstar{}algebras by A. Kumjian and D. Pask in \cite{Kumjian-Pask:C-algebras_directed_graphs} and also by S. Kaliszewski, J. Quigg, and I. Raeburn in \cite{Echterhoff-Quigg:InducedCoactions} and \cite{Kaliszewski-Quigg-Raeburn:Skew_products}, as shortly reviewed in the introduction. 
We begin by formally introducing separated graphs based on \cite{Ara-Goodearl:C-algebras_separated_graphs} and \cite{Ara-Goodearl:Leavitt_path}:   

\begin{definition}
	A \emph{separated graph} is a pair $(E,C)$ where $E$ is a graph and $C=\displaystyle\cup_{v \in E^0} C_v$ in which $C_v$ is a partition of $s^{-1}(v)$ into pairwise disjoint nonempty subsets for each vertex $v$. If all the sets in $C$ are finite, we say that $(E,C)$ is a \emph{finitely separated graph}. This is automatically true when $E$ is row-finite. 
	
	The set $C$ is the \emph{trivial separation} of $E$ if $C_v=\{s^{-1}(v)\}$ for all $v \in E^0$ (in case $v$ is a sink then $s^{-1}(v) = \emptyset$ and therefore we take $C_v$ to be a empty family of subsets). In this case, $(E,C)$ is called a \emph{trivially separated graph} or a \emph{non-separated graph}. Any graph $E$ may be paired with the trivial separation and may thus be viewed as a trivially separated graph. For $X\in C_v$, we set $s(X) = v$. 
\end{definition}

We shall use the following extended notion of graph (iso)morphisms introduced in Section~\ref{sec:DirectedGraphs} to the setting of separated graphs.

\begin{definition}
Let $(E,C)$ and $(F,D)$ be two separated graphs. An (iso)morphism $(E,C)\to (F,D)$ is a graph (iso)morphism $f: E\rightarrow F$ that permutes the separations in the sense that for each $v \in E^0 $ and $X \in C_v$ we have $f^1(X) \in D_{f^0(v)}$. The set of all automorphisms of $(E,C)$ will be denoted by $\Aut (E,C)$; this is a group under composition. 
\end{definition}  

\begin{definition}
	Let $(E,C)$ be a separated graph and $G$ be a group. An {\it action} $\alpha$ of $G$ on $(E,C)$ is a group homomorphism $\alpha\colon G \to \Aut (E,C)$. The action $\alpha $ is said to be {\it free} if the underlying action of $G$ on $E^0$ (and hence also on $E^1$) is free.
\end{definition}

Note that a (free) action of $G$ on $(E,C)$ induces a (free) action of $G$ on $C$.

\begin{example}\label{excuntzgrafoseparado}
	Consider the \emph{Cuntz} graph $A_n$ that has one vertex $v$ and $n$ loops $a_1, \ldots, a_n$ based at $v$ and
	define a separation as $D=D_v:=\{X_1,\ldots,X_n\}$ where $X_i=\{a_i\}$ are a singleton sets for all $i \in \{1, \ldots , n\}$. Thus $(A_n,D)$ is a separated graph, called the Cuntz separated graph, see Figure \ref{fig:cunztgraph} below.
	\begin{center}{
			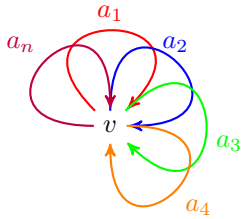
\begin{figure}[htb]
				\begin{tikzpicture}[scale=2]
					\node (v) at (1,0) {$v$};
					\Loop[dist=1cm,dir=NO,label=$a_1$,labelstyle=above,color=red](v)
					\Loop[dist=1cm,dir=NOEA,label=$a_2$,labelstyle=above,color=blue](v)
					\Loop[dist=1cm,dir=EA,label=$a_3$,labelstyle=below right,color=green](v)  \Loop[dist=1cm,dir=SOEA,label=$a_4$,labelstyle=below right,color=orange](v)
					\Loop[dist=1cm,dir=NOWE,label=$a_n$,labelstyle=above left,color=purple](v)
				\end{tikzpicture}
				\caption{The Cuntz separated graph}
				\label{fig:cunztgraph}
		\end{figure}}
	\end{center}
\end{example}			

\begin{example}\label{exemplografoseparado}
	For all integers $1\leq m \leq n$, define the separated graph $(E(m,n),C(m,n))$ as follows: 
	\begin{enumerate}
		\item $E(m,n)^0:=\{v,w\}$ with $v \neq w$,
		\item $E(m,n)^1:=\{e_1,\ldots,e_n,f_1,\ldots,f_m\}$ ($n+m$ distinct edges),
		\item $s(e_i)=s(f_j)=v$ and $r(e_i)=r(f_j)=w$ for all $i,j$,
		\item $C(m,n)=C(m,n)_v:=\{X,Y\}$, where $X=\{e_1,\ldots,e_n\}$ and $Y=\{f_1,\ldots,f_m\}$.
	\end{enumerate}
\begin{center}{
		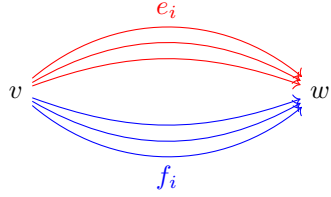
\begin{figure}[htb]
			\begin{tikzpicture}[scale=2]
				\node (v) at (0,0) {$v$};
				\node (w) at (2,0) {$w$};
				\draw[->,red]  (v) to [bend left=40] node [above] {$e_i$} (w);
				\draw[->,red]  (v) to [bend left=30] (w);
				\draw[->,red]  (v) to [bend left=20] (w);
				\draw[->,blue]  (v) to [bend right=40] node [below] {$f_i$} (w);
				\draw[->,blue]  (v) to [bend right=30] (w);
				\draw[->,blue]  (v) to [bend right=20] (w);
			\end{tikzpicture}
			\caption{The separated graph $(E(m,n),C(m,n))$}
			\label{fig:separtedgraphEmn}
	\end{figure}}
\end{center}
This graph, described in Figure \ref{fig:separtedgraphEmn}, admits no free actions of $G$ unless $G$ is the trivial group. The reason is because if $g\neq 1$ then we must have $\alpha_g(v) = w$. Then we have $s(\alpha_g(e_i)) = v \neq w = \alpha_g(s(e_i))$ regardless how the action acts on edges. We will return to this example later. 
\end{example}

An action of $G$ on a separated graph $(E,C)$ yields a quotient separated graph as follows: Keep the usual quotient graph $E/G=(E^0/G,E^1/G,s_G,r_G)$. We are going to define a separation on this graph. For each $X \in C_v$, define $X_G= \{[e]: e\in X\}\subseteq s_G^{-1}([v])$, that is, $X_G$ is the set of equivalence classes of edges which belong to the respective set $X$. The union of $X_G$, for $X\in C_v$, equals  $s_G^{-1}([v])$ since these sets $X$ form a partition of $s^{-1}(v)$. Note that $X_G=Y_G$ if and only if $[X]=[Y]$, where $[X]$ denotes the class of $X\in C$ under the induced action of $G$ on $C$. So, these subsets $X_G$ determine a separation $$C/G:=\displaystyle\bigcup_{[v] \in E^0/G} (C/G)_{[v]}$$ for $E/G$, where $(C/G)_{[v]}:=\{X_G: X\in C_v\}$. Then $(E/G,C/G)$ is called \emph{the quotient separated graph}.

Now, let $(E,C)$ be a separated graph, let $c:E^1 \rightarrow G$ be a labeling function and keep the usual skew product graph $E \times_c G$. For each $X \in C_v$ and $g \in G$ define $X_g:=X \times \{g\} = \{(e,g) ~|~ e \in X\}$ which is a partition of $s^{-1}(v,g)$.  The subsets $X_g$ determine a separation $$C \times_c G := \displaystyle\bigcup_{(v,g) \in E^0 \times G} C \times_c G_{(v,g)}$$ with $C \times_c G_{(v,g)} := \{X_g ~\mid~ X \in C_v\}$. Then $(E\times_c G,C \times_c G)$ is called \emph{the skew product separated graph}.

\begin{theorem}[Gross-Tucker for separated graphs]\label{theo:Gross-Tucker-Separated}
	Suppose a group $G$ acts freely on a separated graph $(E,C)$. Then there is a function $c: E^1/G \to G$ and a $G$-equivariant isomorphism of separated graphs:
	$$(E,C) \cong (E/G\times _cG,C/G \times _c G).$$
\end{theorem}
\begin{proof}
	From the Gross-Tucker Theorem for non-separated graphs  \cite[Theorem 2.2.2]{Gross-Tucker:topolgical-graph-theory}, we already have a labeling function $c: E^1/G \to G$ and a $G$-equivariant isomorphism $\phi: E/G \times_c G \to E$. We only need to show that this isomorphism permutes the separations.  
	
	Consider $x\in E^0/G$ and a base vertex $v_x$ of $x$ in $E^0$. We need to prove that for each $(x,g) \in E^0/G \times_c G$ and $Y \in (C/G \times_c G)_{(x,g)}$ we have $\phi(Y) \in C_{\alpha_g(v_x)}$.  By definition of the separations we have $(C/G \times_c G)_{(x,g)} = (C/G)_{x} \times \{g\}$ and $(C/G)_{x}=\{ X_G : X\in C_{v_x}\}$.  For all $(y,g) \in X_G\times \{ g\}$, where $X\in C_{v_x}$, we have $\phi(y,g) = \alpha_g(e_y)$, where $e_y\in X$ satisfies $[e_y]=y$. It follows that 
	$\phi (X_G\times \{g\})= \alpha_g (X) \in C_{\alpha_g(v_x)}$, as desired. 
	\end{proof}

The above result extends the original Gross-Tucker Theorem \cite[Theorem 2.2.2]{Gross-Tucker:topolgical-graph-theory} for ordinary graphs and it is strongly related to a similar result obtained for labeled graphs in \cite{Bates-Pask-Willis:Group_actions}. 

\begin{example}\label{exgrafodecayleyseparado} 
Let $G$ be a group with generators $g_1, \ldots, g_n$. The \emph{Cayley graph} of $G$ with respect to $g_1, \ldots, g_n$ is the graph $E_G:=(E_G^0,E_G^1,s,r)$, where $E_G^0 = G$, $E_G^1= G \times \{ g_1,\ldots, g_n\}$ and the source and range maps are defined by $s(h,g_i) = h  \text{ and } r(h,g_i) = hg_i$ for all $i \in \{1,\ldots,n\}$. This graph carries a natural free action of $G$ by left translations:  $g\cdot (h,g_i)=(gh,g_i)$. By the Gross-Tucker Theorem, it is a skew product graph. Indeed, the quotient $E_G/G$ identifies with the Cuntz graph $A_n$. Using the labeling function $c: A_n^1 \rightarrow G$ with $c(a_i) = g_i$ for all $i$, we get a $G$-equivariant isomorphism of graphs 
	$$E_G \cong (E_G/G) \times _c G \cong A_n \times_c G.$$
Now we introduce a separation  on $E_G$  as follows: $C_G = \displaystyle\cup_{g \in G} (C_G)_g$ where $(C_G)_g := \{X_1^g,\ldots,X_n^g\}$ in which $X_i^g := \{g\} \times \{g_i\}$ for every $i \in \{1,\ldots,n\}$. This yields a separated graph $(E_G,C_G)$, called the Cayley separated graph. 	

The canonical free action on $E_G$ permutes the separations and hence yields a free action on $(E_G,C_G)$. By the separated version of the Gross-Tucker Theorem~\ref{theo:Gross-Tucker-Separated}, $(E_G,C_G)$ is isomorphic to a separated skew product graph. Indeed, we can identify it with $(A_n \times_c G, D \times_c G)$, where $D$ is the separation on $A_n$  defined in Example~\ref{excuntzgrafoseparado}.
\end{example}

\subsection{$C^*$-algebras of separated graphs and duality theorems}\label{sec:duality-full-separated}

In this section, we are going to formally define the separated graph \cstar{}algebras based on \cite{Ara-Goodearl:C-algebras_separated_graphs} and \cite{Ara-Goodearl:Leavitt_path}, and extend some of the duality theorems from ordinary graphs as explained in the introduction. 

\begin{definition}\label{defC*grafosep}
The \emph{Leavitt path algebra} of a separated graph $(E,C)$ is the universal complex *-algebra $L(E,C)$ with generators $\{P_v\}_{v \in E^0}$ of mutually orthogonal projections and $\{S_e\}_{e \in E^1}$ of partial isometries subject to the following relations: 
	\begin{enumerate}
		\item $P_{s(e)}S_e = S_eP_{r(e)} = S_e$ for all $e \in E^1$,
		\item $S_e^*S_f = \delta_{e,f} P_{r(e)}$ for all $e,f \in X$, $X \in C$,
		\item $P_{v} = \displaystyle\sum_{e \in X} S_eS_e^*$ for every finite subset $X \in C_v$.
	\end{enumerate}
\end{definition}

\begin{definition}
	The \emph{\cstar{}algebra} of a separated graph $(E,C)$ is the universal \cstar{}algebra $\CEC$ with generators $\{P_v,S_e ~|~ v \in E^0, e \in E^1 \}$ subject to the relations (1)-(3) of Definition \ref{defC*grafosep}. The collection $\{P_v,S_e ~|~ v \in E^0, e \in E^1 \}$ is called a \emph{universal Cuntz-Krieger $(E,C)$-family}. In other words, $\CEC$ is the enveloping \cstar{}algebra of $L(E,C)$. 
\end{definition}


\begin{definition}
	For two non-trivial paths $\mu,\nu \in \text{Path}(E)$ with $s(\mu)=s(\nu)=v$ we say that $\mu$ and $\nu$ are \emph{$C$-separated} if the initial edges of $\mu$ and $\nu$ belong to different sets $X,Y \in C_v$. 
\end{definition} 

\begin{definition}
	For each finite $X \in C$, we select an edge $e_X \in X$. Let $\mu,\nu \in \text{Path}(E)$ be two paths such that $r(\mu)=r(\nu)$ and let $e$ and $f$ be the terminal edges of $\mu$ and $\nu$, respectively. The path $\mu \nu^*$ is said to be \emph{reduced} if $(e,f) \neq (e_X,e_X)$ for every finite $X \in C$. In case either $\mu$ or $\nu$ has length zero then $\mu\nu^*$ is automatically reduced.  
\end{definition}

\begin{remark}
	Every time we use the notion of reduced paths the choice above is applied.
\end{remark}

We will repeatedly use the following result:

\begin{proposition}\cite[Corollary 2.8]{Ara-Goodearl:Leavitt_path}\label{basegrafosseparados}
	Let $(E,C)$ be a separated graph. Then the set of elements of the form $$S_{\mu_1}S_{\nu_1}^*S_{\mu_2}S_{\nu_2}^* \ldots S_{\mu_n}S_{\nu_n}^*, ~~ \mu_i,\nu_i \in \text{Path}(E)$$ such that  $\nu_i$ and $\mu_{i+1}$ are $C$-separated paths for all $i \in \{1,\ldots,n-1\}$ and $\mu_i\nu_i^*$ is reduced for all $i \in \{1,\ldots,n\}$ forms a linear basis of $L(E,C)$. We call $\mu_1\nu_1^* \ldots \mu_n\nu_n^*$ a $C$-separated reduced path. 
	\end{proposition}

\begin{example}\label{excanonicografoseparado1}
	Consider the separated graph $(E(m,n),C(m,n))$ seen in Example \ref{exemplografoseparado}. In this example an element of the form $S_{e_i}^*S_{f_j}$ is not zero because $e_i$ and $f_j$ are in different sets for every $i,j$. In the context of ordinary (unseparated) graphs, all these elements are required to be zero. It is proved in \cite{Ara-Goodearl:C-algebras_separated_graphs} that 
$$C^*(E(m,n),C(m,n)) \cong M_{n+1}(U_{m,n})\cong M_{m+1}(U_{m,n}),$$ 
where $U_{m,n}$ is the universal \cstar{}algebra generated by the entries of a unitary $m \times n$ matrix, originally studied by Brown in \cite{Brown:Ext_free} and more generally by McClanahan in \cite{McClanahan:K-theory}. 
	For the special case of $m=1$ and $n\geq 2$ we get the Cuntz \cstar{}algebra $U_{1,n} \cong \mathcal{O}_n$ and consequently 
	$$C^*(E(1,n),C(1,n)) \cong M_{2}(\mathcal{O}_n).$$
In addition, $C^*(E(1,1),C(1,1)) \cong M_{2}(C(\mathbb{T}))$.
\end{example}

\begin{example}\label{excanonicografoseparado}
Another important basic example is given by the Cuntz separated graph $(A_n,D)$ from Example \ref{excuntzgrafoseparado}. In this case $C^*(A_n,D) \cong C^*(\mathbb{F}_n)$ via an isomorphism that identifies the generator $S_{a_i}\in C^*(A_n,D)$ with $a_i\in \mathbb{F}_n$ viewed as an element of $C^*(\mathbb{F}_n)$. Here $\mathbb{F}_n$ denotes the free group generated by the edges $a_1,\ldots,a_n$. More generally, we can consider any graph $E$ with only one vertex $v$ and the separation $C$ where all sets are singletons $X=\{e\}$ with $e\in E^1$. In this case $C^*(E,C)\cong C^*(\Free)$, the full \cstar{}algebra of the free group $\Free=\Free_{E^1}$ on the set $E^1$.
\end{example}

In order to deal with actions of groups, it is interesting to work with a class of normal forms which is closed under the action of the group. This description uses only purely multiplicative expressions and is analogous to the one considered in \cite{meakin-milan-wang-2021}, see also \cite{fan-wang-2021}.

 \begin{definition}
 	\label{def:weaklyreducedpaths} Let $(E,C)$ be a separated graph.
 	 A path
$\mu= \mu_1 \nu_1^* \cdots \mu_n\nu_n^*\in \text{Path}(\hat{E})$ in the extended graph $\hat{E}$ is said to be a {\it $C$-separated weakly reduced path} in case it is $C$-separated and does not contain subpaths of the form $ee^*$, where $\{e\}\in C$.
\end{definition}

    Observe that if $(A_n,D)$ is the Cuntz graph as in Example \ref{excuntzgrafoseparado}, then the set $\mathcal S$ of $D$-separated weakly reduced paths is indeed the set of reduced words of the free group, in the usual sense.   

\begin{lemma}
	If $G$ acts on a separated graph $(E,C)$, then there is an induced action $\alpha: G \rightarrow Aut(\CEC)$ such that $\alpha_g(S_e) = S_{g \cdot e}$ e $\alpha_g(P_v) = P_{g \cdot v}$ for all $e \in E^1 $ and $v\in E^0$.
\end{lemma}

\begin{proof}
	Fix $g \in G$ and define $P_v' := P_{g\cdot v}$ and $S_e' := S_{g\cdot e}$ for all $v \in E^0$ and $e \in E^1$, where $\cdot$ denotes the action of $G$ on $(E,C)$. We claim that $\{P_v',S_e'\}$ is a Cuntz-Krieger $(E,C)$-family in $\CEC$. Note that for all $v,w \in E^0$ we have $g\cdot v = g \cdot w $ if and only if $v=w$. Similarly for edges. It is trivial to see that $\{P_v'\}$ is a family of mutually orthogonal projections satisfying Definition~\ref{defC*grafosep}(1). To see condition (2), note that for $e,f \in X$, $X \in C_v$ we have 
	$$	S_e'^*S_f' =  S_{g\cdot e}^*S_{g\cdot f} = \delta_{g\cdot e,g\cdot f} P_{r(g\cdot e)} = \delta_{e,f} P_{gr(e)} = \delta_{e,f} P_{r(e)}'.$$
	To see condition (3) note that for all finite $X \in C_v$ we have $g \cdot X \in C_{g \cdot v} $ for all $g\in G$. Then we have    
	$$\displaystyle\sum_{e \in X}S_e'S_e'^* =  \displaystyle\sum_{e \in X}S_{g\cdot e}S_{g\cdot e}^* = \displaystyle\sum_{e \in g\cdot X}S_{e}S_{e}^* = P_{g\cdot v} = P_v' .$$ The universal property of $\CEC$ yields a *-homomorphism $\alpha_g: \CEC \to \CEC$ such that $\alpha_g(S_e) = S_{g\cdot e}$ and $\alpha_g(P_v) = P_{g\cdot v}$ for all $e \in E^1 $ and $v\in E^0$. It is straightforward to check that $\alpha_g^{-1} = \alpha_{g^{-1}}$ and, checking on generators it also follows that $\alpha_g \circ \alpha_h = \alpha_{gh}$ for every $g,h \in G$. Since $\alpha_1 = id$ it follows that $\alpha$ is an action, as desired.  
	\end{proof}

Now we are ready to extend some of the duality results from ordinary graphs to separated graph \cstar{}algebras.

\begin{theorem}\label{isomordeltamaximal}
Let $(E,C)$ be a separated graph with universal Cuntz-Krieger $(E,C)$-family $\{S_e,P_v\}_{e\in E^1, v\in E^0}$.
Given a labeling function $c\colon E^1\to G$ on $(E,C)$, there is a unique  coaction $\delta_c\colon C^*(E,C)\to C^*(E,C)\otimes C^*(G)$ satisfying
	$$\delta_c(P_v)=P_v\otimes 1\quad\mbox{and}\quad \delta_c(S_e)=S_e\otimes c(e)$$
	for all $v\in E^0$ and $e\in E^1$. Moreover, the coaction $\delta_c$ is maximal, that is, we have a canonical isomorphism
	$$C^*(E,C)\rtimes_{\delta_c} G \rtimes_{\widehat{\delta_c}}G\cong C^*(E,C)\otimes \K(\ell^2 G).$$
\end{theorem}
\begin{proof}
To prove existence and uniqueness of $\delta_c$, it is enough to check that the family  $\{P_v \otimes 1,S_e \otimes c(e)\}_{v\in E^0, e\in E^1}$ satisfies conditions (1)-(3) of Definition~\ref{defC*grafosep}. 
It is straightforward to check that $\{P_v \otimes 1\}$ are mutually orthogonal projections satisfying condition (1). Condition (2) also follows easily since
$$(S_e \otimes c(e))^*(S_f \otimes c(f)) = (S_e^* \otimes c(e)^{-1})(S_f \otimes c(f))= S_e^*S_f \otimes c(e)^{-1} c(f) = \delta_{e,f} P_{r(e)} \otimes 1. $$
For the last condition, note that for all finite $X \in C_v$ we have 
	\begin{align*}
		\displaystyle\sum_{e \in X} (S_e  & \otimes c(e))(S_e \otimes c(e))^* = \displaystyle\sum_{e \in X} (S_e \otimes c(e))(S_e^* \otimes c(e)^{-1}) \\
		&=\displaystyle\sum_{e \in X} S_eS_e^* \otimes c(e)c(e)^{-1} 
=\left(\displaystyle\sum_{e \in X} S_eS_e^*\right) \otimes 1 =P_v \otimes 1.
\end{align*}	
So the universal property yields a *-homomorphism $\delta_c: \CEC \rightarrow \CEC \otimes \CG $ satisfying the properties of the statement. To prove the coaction identity, we only need to check it on the generators
$P_v$, $S_e$. On the $P_v$'s it acts trivially $P_v\mapsto P_v\otimes 1$, so nothing happens. And on the $S_e$'s it acts by $S_e\mapsto S_e\otimes c(e)$, in which case  
	\begin{align*}
		(\delta_c \otimes\id_G) \circ \delta_c (S_e) &= (\delta_c \otimes\id_G)(S_e \otimes c(e)) = \delta_c(S_e) \otimes c(e)\\ 
& = (S_e \otimes c(e))\otimes c(e) 
= S_e \otimes (c(e)\otimes c(e))\\ &  = S_e \otimes \delta_G(c(e)) = id \otimes \delta_G(S_e \otimes c(e)) \\
		&= (\id_ \otimes \delta_G) \circ \delta_c(S_e).
	\end{align*} 
To finish the proof that $\delta_c$ is a coaction, we need to check its nondegeneracy, that is,  $\delta_c(\CEC)(1 \otimes \CG)$ spans a dense subspace of $\CEC \otimes \CG$. But by Proposition~\ref{basegrafosseparados}, the elements of the form $S_{\mu_1}S_{\nu_1}^* \ldots S_{\mu_n}S_{\nu_n}^* \otimes g$, where $\mu_1\nu_1^* \cdots \mu_n\nu_n^*$ is a $C$-separated reduced path and $g\in G$, span a dense subspace of $\CEC \otimes \CG$. And for $h:=c(\mu_1)c(\nu_1)^{-1}\ldots c(\mu_n)c(\nu_n)^{-1}$ we observe that 
	\begin{align*}
		S_{\mu_1}S_{\nu_1}^* \ldots S_{\mu_n}S_{\nu_n}^* \otimes g &= (S_{\mu_1}S_{\nu_1}^* \ldots S_{\mu_n}S_{\nu_n}^* \otimes h)(1 \otimes h^{-1}g)\\&= \delta_c(S_{\mu_1}S_{\nu_1}^* \ldots S_{\mu_n}S_{\nu_n}^*)(1 \otimes h^{-1}g).
	\end{align*}
Thus  $\delta_c(\CEC)(1 \otimes \CG)$ contains all elements of the form $S_{\mu_1}S_{\nu_1}^* \ldots S_{\mu_n}S_{\nu_n}^* \otimes g$ which span a dense subspace of $\CEC \otimes \CG$. Therefore $\delta_c$ is a coaction, as desired.

Finally, to see that $\delta_c$ is maximal, we are going to use \cite{Buss-Echterhoff:Maximality}*{Theorem~5.1}, that is, we prove that $\delta_c$ admits a lift to a homomorphism $\tilde\delta_c\colon C^*(E,C)\to C^*(E,C)\otimes_\max C^*(G)$. Indeed, the same proof above that we used for $\delta_c$ also shows that there is a (unique) \Star{}homomorphism $\tilde\delta_c\colon C^*(E,C)\to C^*(E,C)\otimes_\max C^*(G)$ satisfying $\tilde\delta_c(P_v)=P_v\otimes 1$ and $\tilde\delta_c(S_e)=S_e\otimes c(e)$. This is a lift for $\delta_c$ and it is therefore a maximal coaction.
\end{proof}

\begin{theorem}\label{isomorgrafosskewseparproducruzado}
	With notations as above, there is a canonical isomorphism
	$$C^*(E\times_c G,C\times_c G)\cong C^*(E,C)\rtimes_{\delta_c} G .$$
	Under this isomorphism, the $G$-action $\gamma$ on $C^*(E\times_c G,C\times_c G)$ induced by the translation $G$-action on $(E\times_c G,C\times_c G)$
	corresponds to the dual action $\widehat{\delta_c}$ on $C^*(E,C)\rtimes_{\delta_c}G$.
\end{theorem}
\begin{proof}
	Let $B:=\CEC$ and $D:=\CEGC$, and let $(j_B,j_G^B)$ be the canonical covariant representation of $(B,C_0(G))$ in $\M(B \rtimes_{\delta_c} G)$. Let $\{P_v,S_e\}$ and $\{P_{(v,g)},S_{(e,g)}\}$ be the universal Cuntz-Krieger $E$-families for $B$ and $D$, respectively. To define a \Star{}\nb{}homomorphism from the \cstar{}algebra of the separated skew product graph to the crossed product as above let us define $$p_{(v,g)} := j_B(P_v)j_G(\chi_{g^{-1}}) \,\,  \text{ and } \, \, s_{(e,g)} := j_B(S_e)j_G(\chi_{(gc(e))^{-1}})$$ for all $v \in E^0$, $e \in E^1$ and $g \in G$. We claim that $\{p_{(v,g)},s_{(e,g)}\}$ is a Cuntz-Krieger $(E\times _c G, C\times_cG)$-family.
	
	Before we continue, note that the covariance condition in Definition \ref{condicaodecovariancia} and the fact that each $P_v \in B_1$ and $S_e \in B_{c(e)} $ imply 
	\begin{equation}\label{condicaodem}
		j_B(P_v)j_G(\chi_{g^{-1}}) = j_G(\chi_{g^{-1}})j_B(P_v) \text{ and } j_B(S_e)j_G(\chi_{(gc(e))^{-1}}) = j_G(\chi_{g^{-1}})j_B(S_e).
	\end{equation} 
	To check the conditions in Definition~\ref{defC*grafosep} we use the covariance condition and~\eqref{condicaodem}:
	\begin{align*}
		p_{(v,g)}p_{(w,h)} &= j_B(P_v)j_G^B(\chi_{g^{-1}})j_B(P_w)j_G^B(\chi_{h^{-1}})\\
		&=j_B(P_v)j_B(P_w)j_G^B(\chi_{g^{-1}})j_G^B(\chi_{h^{-1}})  \text{ (by \eqref{condicaodem})}\\
		&=j_B(P_vP_w)j_G^B(\chi_{g^{-1}}\chi_{h^{-1}})  \\
		&=\delta_{(v,g), (w,h)} p_{(v,g)}.
	\end{align*} 
	Thus $p_{(v,g)}$ is a family of mutually orthogonal projections. The first condition is
	\begin{align*}
		p_{s(e,g)}s_{(e,g)} &= j_B(P_{s(e)})j_G^B(\chi_{g^{-1}})j_B(S_e)j_G^B(\chi_{(gc(e))^{-1}})\\
		&=j_B(P_{s(e)})j_B(S_e)j_G^B(\chi_{(gc(e))^{-1}})j_G^B(\chi_{(gc(e))^{-1}}) \text{ (by \eqref{condicaodem})}\\
		&=j_B(P_{s(e)} S_e)j_G^B(\chi_{(gc(e))^{-1}})\\
		&=j_B(S_e)j_G^B(\chi_{(gc(e))^{-1}}) \\
		& = s_{(e,g)}.
	\end{align*}
	The second condition: for $(e,g),(f,g) \in X_g = X\times \{g\}$, $X \in C$,  
	\begin{align*}
		s_{(f,g)}^*s_{(e,g)} &= (j_B(S_f)j_G^B(\chi_{(gc(f))^{-1}}))^*j_B(S_e)j_G^B(\chi_{(gc(e))^{-1}}) \\
		&= j_G^B(\chi_{(gc(f))^{-1}})j_B(S_f^*)j_B(S_e)j_G^B(\chi_{(gc(e))^{-1}}) \\
		&= \delta_{e,f} j_G^B(\chi_{(gc(e))^{-1}})j_B(P_{r(e)}) j_G^B(\chi_{(gc(e))^{-1}}) \\
		& = j_B(P_{r(e)})j_G^B(\chi_{(gc(e))^{-1}}) 
		= p_{(r(e),gc(e))} = p_{r(e,g)}.
			\end{align*}
	And the third and last condition: for finite $X \in C_v$, $v \in E^0$ the subset $X_g=X\times \{g\} \in (C \times_c G)_{(v,g)}$ is finite and we have 
	\begin{align*}
		\displaystyle\sum_{(e,g) \in X_g} s_{(e,g)}s_{(e,g)}^* &= \displaystyle\sum_{(e,g) \in X_g}j_B(S_e)j_G^B(\chi_{(gc(e))^{-1}})(j_B(S_e)j_G^B(\chi_{(gc(e))^{-1}}))^* \\
		&=\displaystyle\sum_{(e,g) \in X_g}j_B(S_e)j_G^B(\chi_{(gc(e))^{-1}})j_B(S_e^*)  \\
		&=\displaystyle\sum_{(e,g) \in X_g}j_B(S_e)j_B(S_e^*)j_G^B(\chi_{g^{-1}})\quad  (\text{by \ref{condicaodem}}) \\
		&=j_B\left(\displaystyle\sum_{e \in X} S_eS_e^*\right)j_G^B(\chi_{g^{-1}}) 
		=j_B(P_v)j_G^B(\chi_{g^{-1}})  = p_{(v,g)}.
	\end{align*}
	By the universal property there is a *-homomorphism $$\phi: \CEGC \rightarrow \CEC \rtimes_{\delta_c} G $$ satisfying $\phi(P_{(v,g)}) = p_{(v,g)}$ and $\phi(S_{(e,g)}) = s_{(e,g)}$. To see that it is surjective, we first note that for every path $\mu \in \text{Path}(E)$ and $g \in G$ we have a unique path $(\mu,g)$ in $E \times_c G$ such that $\phi(S_{(\mu,g)}) = s_{(\mu,g)} =  j_G^B(\chi_{g^{-1}})j_B(S_\mu)$. By Proposition \ref{basegrafosseparados} we have basis elements in $L(E\times_c G,C \times_c G)$ of the form: 
	\begin{equation}
	\label{eq:pathinCtimesG} S_{(\mu_1,g)}S_{(\nu_1,z_1)}^*S_{(\mu_2,z_1)}S_{(\nu_1,z_2)}^* \ldots S_{(\mu_n,z_{n-1})}S_{(\nu_n,z_n)}^*
	\end{equation}
	where $(\mu_1,g)(\nu_1,z_1)^*(\mu_2,z_1)(\nu_2,z_2)^{*}\cdots (\mu_n,z_{n-1}) (\nu_n,z_n)^*$ is a $C\times G$-separated reduced path. Note that $\varsigma:= \mu_1\nu_1^* \cdots \mu_n\nu_n^*$ is a $C$-separated reduced path, and that $z_1 = g c(\mu_1)c(\nu_1)^{-1}$  and $z_i =z_{i-1}c(\mu_i)c(\nu_i)^{-1}$ for all $i \in \{2,\ldots,n\}$. Writing $$(\varsigma, g) := (\mu_1,g)(\nu_1,z_1)^*\cdots (\mu_n,z_{n-1}) (\nu_n,z_n)^*,$$ we observe that $s(\varsigma,g)= (s(\varsigma), g)$ and $r(\varsigma ,g)= (r(\varsigma ), gc(\varsigma))$. 
	
	Now, we are going to compute $\phi$ on the basis elements of the form \eqref{eq:pathinCtimesG}. 
	Denoting by $S_{(\varsigma, g)}$ such element, taking into account the above computations and using  \eqref{condicaodem},  we get	
	\begin{align*}
		\phi(S_{(\varsigma ,g)}) & 
    = j_B(S_{\mu_1})j_G^B(\chi_{(gc(\mu_1))^{-1}}) j_B(S_{\nu_1})^*j_B(S_{\mu_2})j_G^B(\chi_{(z_1c(\mu_2))^{-1}})j_B(S_{\nu_2})^*\cdots \\
	 & \cdots  j_B(S_{\mu_n})j_G^B(\chi_{(z_{n-1}c(\mu_n))^{-1}})j_B(S_{\nu_n})^* \\
		&= j_G^B(\chi_{g^{-1}})j_B(S_{\varsigma}).
	\end{align*}
	Extending linearly to $L(E\times_c G,C \times_c G)$ it follows that the image of $\phi$ contains all elements of the form $j_G^B(\chi)j_B(b)$ with $\chi \in C_c(G)$ and $b \in L(E,C)$. Since the span of these elements is a dense subspace of the crossed product $B \rtimes_{\delta_c} G$ this shows that $\phi$ is surjective.   
	
To prove that $\phi$ is injective and hence an isomorphism we cannot follow the same idea as for non-separated graphs in \cite[Theorem 2.4]{Kaliszewski-Quigg-Raeburn:Skew_products} because here we do not have an injectivity theorem as used for graphs. So the way to show this is to construct the inverse. For this, we are going to define a covariant representation $(\pi,\sigma)$ of $(\CEC,G,\delta_c)$ into $\M(\CEGC)$ which will be given by: 
	\begin{align}\label{isomorgrafosskewseparproducruzadoinversa}
		\pi(P_v) = \displaystyle\sum_{g  \in G} P_{(v,g)}, \quad \pi(S_e) = \displaystyle\sum_{g  \in G} S_{(e,g)} \quad \text{ and } \quad \sigma(\chi_g) = \displaystyle\sum_{v  \in E^0} P_{(v,g^{-1})}
	\end{align}  for all $v \in E^0$, $e \in E^1$ and $g \in G$. First of all, we need to make sure that the above formulas define *-homomorphisms. For $v \in E^0$, we claim that the sum $\sum_{g  \in G} P_{(v,g)}$ converges in the strict topology of $\M(\CEGC)$. Indeed, since the net of all finite sums of projections has norm uniformly bounded by 1, 
it is enough to check that $(\sum_{g} P_{(v,g)}) S_{(\varsigma,g)}$ 
	converges for each $C\times G$-separated reduced path $(\varsigma, g)$, where we follow the conventions and notations introduced above. But it is easily seen that this converges to either $S_{(\varsigma, g)}$ or $0$, depending on whether $v=s(\varsigma)$ or not.  
	
	So, we have a well-defined element $\pi(P_v) \in \M(\CEGC)$ which is the strict limit of the net considered above. A similar argument shows that, for each $e \in E^1$, $\pi(S_e)$ yields a well-defined element of $\M(\CEGC)$. It is straightforward to check that $\{\pi(P_v), \pi(S_e)\}$ is a Cuntz-Krieger $(E,C)$-family. By the universal property there exists a $\pi$ with the desired properties \eqref{isomorgrafosskewseparproducruzadoinversa}. The formula for $\sigma$ in \eqref{isomorgrafosskewseparproducruzadoinversa} also makes sense (the sum converges strictly) and the family of projections $\{\sigma(\chi_g): g\in G\}$ is mutually orthogonal, hence it determines a *-homomorphism $\sigma: C_0(G) \to \M(\CEGC)$.
	
	Now, we are going to check that $(\pi, \sigma)$ is a covariant representation, that is, we prove the relation in Definition \ref{condicaodecovariancia}. To see this, fix a $C$-separated reduced path  $\varsigma $ and $g \in G$. Since $S_\varsigma \in B_{c(\varsigma)}$, on the one hand we have 
	\begin{align*}
		\pi(S_\varsigma)\sigma(\chi_g) &= \left(\displaystyle\sum_{h  \in G} S_{(\varsigma,h)}\right)\left(\displaystyle\sum_{v  \in E^0} P_{(v,g^{-1})}\right)\\
		&= \displaystyle\sum_{h  \in G, v \in E^0} S_{(\varsigma ,h)}P_{(v,g^{-1})}\\
		&=S_{(\varsigma,(c(\varsigma )g)^{-1})} 
	\end{align*}
	where in the last step we have used that the summand $S_{(\varsigma ,h)}P_{(v,g^{-1})}$ is non-zero if and only if $(v,g^{-1})=r(\varsigma,h) = (r(\varsigma),hc(\varsigma))$, that is, $v=r(\varsigma)$ and $h=(c(\varsigma)g)^{-1}$. On the other hand, 
	\begin{align*}
		\sigma(\chi_{c(\varsigma)g})\pi(S_\varsigma) &= \left(\displaystyle\sum_{v  \in E^0} P_{(v,(c(\varsigma)g)^{-1})}\right)\left(\displaystyle\sum_{h  \in G} S_{(\varsigma,h)}\right)\\
		&= \displaystyle\sum_{h  \in G, v \in E^0} P_{(v,(c(\varsigma)g)^{-1})}S_{(\varsigma,h)}\\
		&=S_{(\varsigma,(c(\varsigma)g)^{-1})} 
	\end{align*}
	where again we used that the only non-zero summand is for $(v,(c(\varsigma)g)^{-1})=s(\varsigma,h) = (s(\varsigma),h)$, that is, $v=s(\varsigma)$ and $h=(c(\varsigma)g)^{-1}$. This proves the covariance relation for the pair $(\pi,\sigma)$ and therefore by the universal property we get a nondegenerate *-homomorphism $$\psi:=\pi \times \sigma: \CEC \rtimes_{\delta_c} G \rightarrow \M(\CEGC)$$ such that $\psi \circ j_B  = \pi$ and $\psi \circ j_G^B = \sigma$. Now, we compute: 
		\begin{align*}
		\psi \circ \phi ( P_{(w,g)}) &= \psi(p_{(w,g)})
		= \psi(j_B(P_w)j_G^B(\chi_{g^{-1}})) \\
		&=\pi(P_w)\sigma(\chi_{g^{-1}}) 
		=\left(\displaystyle\sum_{h  \in G} P_{(w,h)}\right)\left(\displaystyle\sum_{v  \in E^0} P_{(v,g)}\right) \\
		&= \displaystyle\sum_{h  \in G, v \in E^0} P_{(w,h)}P_{(v,g)} =P_{(w,g)}
	\end{align*} 
	and 
	\begin{align*}
		\psi \circ \phi ( S_{(e,g)}) &= \psi(s_{(e,g)}) = \psi(j_B(S_e)j_G^B(\chi_{(gc(e))^{-1}})) \\
		&=\pi(S_e)\sigma(\chi_{(gc(e))^{-1}}) 
		=\left(\displaystyle\sum_{h  \in G} S_{(e,h)}\right)\left(\displaystyle\sum_{v  \in E^0} P_{(v,gc(e))}\right)\\ 
		& = \displaystyle\sum_{h  \in G, v \in E^0} S_{(e,h)}P_{(v,gc(e))} 
		=S_{(e,g)}.
	\end{align*} 
	Since the elements $P_{(v,g)}$ and $S_{(e,g)}$ generate the \cstar{}algebra $\CEGC$, it follows that $\psi \circ \phi = id$ and hence $\phi$ is injective, therefore an isomorphism. In particular, the range of $\psi$ is inside of $\CEGC$. Finally, we are going to check the $G$-equivariance of the actions. Note that for all $z\in G$,
	\begin{align*}
		\phi(\gamma_z(P_{(v,g)})) &= \phi(P_{(v,zg)}) 
		= p_{(v,zg)} 
=j_B(P_v)j_G^B(\chi_{g^{-1}z^{-1}}) \\
&		=(\widehat{\delta_c})_z(j_B(P_v)j_G^B(\chi_{g^{-1}})) 
        =(\widehat{\delta_c})_z(\phi(P_{(v,g)})) 
	\end{align*}
	and
	\begin{align*}
		\phi(\gamma_z(S_{(e,g)})) &= \phi(S_{(e,zg)}) 
		= s_{(e,zg)} 
		=j_B(S_e)j_G^B(\chi_{(zgc(e))^{-1}}) \\
		&=(\widehat{\delta_c})_z(j_B(S_e)j_G^B(\chi_{(gc(e))^{-1}})) 
		=(\widehat{\delta_c})_z(\phi(S_{(e,g)})).
	\end{align*}	
\end{proof}

\begin{corollary}\label{corolariografosskewseparados}
	For a separated graph $(E,C)$ and a labeling $c: E\to G$, there is a canonical isomorphism
	$$C^*(E\times_c G,C\times_c G)\rtimes_{\gamma}G\cong C^*(E,C)\otimes\K(\ell^2 G)$$
	where $\gamma$ is the action of $G$ on $C^*(E\times_c G,C\times_c G)$ induced by the translation action on $(E\times_c G,C\times_c G)$.
\end{corollary}

\begin{proof}
	Follows from Theorems \ref{isomorgrafosskewseparproducruzado} and \ref{isomordeltamaximal}. 
\end{proof}

\begin{corollary}\label{corolariografosskewseparados1}
	For a free action $\theta$ of a group $G$ on a separated graph $(E,C)$, there is a canonical isomorphism
	$$C^*(E,C)\rtimes_{\theta}G\cong C^*(E/G,C/G)\otimes\K(\ell^2 G).$$
\end{corollary}

\begin{proof}
	Follows from Corollary~\ref{corolariografosskewseparados} and the Gross-Tucker Theorem~\ref{theo:Gross-Tucker-Separated}.
\end{proof}

\begin{example}
	If we consider the Cayley separated graph $(E_G,C_G)$ from Example \ref{exgrafodecayleyseparado} we know that  $(E_G,C_G)$ carries a free action $\beta$ of $G$ and hence by Corollary \ref{corolariografosskewseparados1} we see that $$C^*(E_G,C_G) \rtimes_{\beta} G \cong C^*(A_n,D) \otimes \K(\ell^2 G) \cong C^*(\F_n) \otimes \K(\ell^2 G),$$ 
where $\F_n$ is the free group generated by the $n$ edges of $A_n$. 
\end{example}

\begin{remark}\label{grafoCuntzseparado}
	The coaction $\delta_c$ is not always a normal coaction on $\CEC$. For example, if we consider the separated graph $(A_n,D)$ from Example \ref{excanonicografoseparado} and the label $c$ from $A_n$ to the free group $G=\Free_n$ on the $n$ edges of $A_n$, then we have $C^*(A_n,D) \cong C^*(\mathbb{F}_n)$ and $\delta_c$ coincides with the canonical coaction $\delta_{\Free_n}$ of $\Free_n$ on its \cstar{}algebra $C^*(\Free_n)$ given by its comultiplication. Then $\delta_{\mathbb{F}_n}$ is maximal but not a normal coaction since $\mathbb{F}_n$ is not amenable for $n>1$. The normalization of $\delta_{\mathbb{F}_n}$ is the canonical coaction on $C_r^*(\mathbb{F}_n)$. This is related to the reduced \cstar{}algebra of $(E,C)$, to be introduced in the next section.
\end{remark}

Summarising the main results from this section, we obtain the following commutative diagram of isomorphisms:
	\[
	\begin{tikzcd}
		\CEGC \rtimes_{\gamma} G \arrow{dr}{\ref{corolariografosskewseparados}} \arrow{rr}{\ref{isomorgrafosskewseparproducruzado}}  & &  \CEC \rtimes_{\delta_c} G \rtimes_{\hat{\delta_c}} G  \arrow{dl}{\ref{isomordeltamaximal}} \\
		&\CEC \otimes \K(\ell^2 G)  \\
	\end{tikzcd}
	\]

\section{Reduced Separated Graph $C^*$-algebras}\label{cap3}

In this section we deal with reduced \cstar{}algebras of separated graphs. For this purpose, we need to restrict our attention to finitely separated graphs. 
The main reason for this restriction is Proposition~\ref{esperançacondicional1} on the existence of a certain canonical conditional expectation on usual graph \cstar{}algebras taking values onto the commutative \cstar{}algebra generated by the vertex projections of the graph. The existence of this conditional expectation is only granted in the row-finite case. There seems to be no such expectation in general, as witnessed for instance by the graph with only one vertex and infinitely many edges, which gives rise to the Cuntz algebra $\Cuntz_\infty$.

The original definition of reduced \cstar{}algebras uses reduced amalgamated free products, see \cite{Voiculescu:Symmetries} and \cite{Ara-Goodearl:C-algebras_separated_graphs}. Our strategy here is to give an alternative description of reduced \cstar{}algebras of separated graphs that emphasizes a canonical conditional expectation $P\colon C^*(E,C) \rightarrow C_0(E^0)$.  

\subsection{Reduced amalgamated free products}

Let $(E,C)$ be a separated graph. For each $X \in C$, we consider the subgraph $E_X$ of $E$ with $(E_X)^0 := E^0$ and $(E_X)^1 := X$. 
The canonical inclusions $(E^0,\emptyset) \hookrightarrow (E_X,\{ X \}) \hookrightarrow (E,C)$ yield \cstar{}algebra inclusions
$$C_0(E^0) = C^*(E^0,\emptyset)\into C^*(E_X)\into \CEC.$$

We then recall the following result:

\begin{proposition}\cite[Proposition 3.1]{Ara-Goodearl:C-algebras_separated_graphs}\label{propcoproduto}
Let $(E,C)$ be a separated graph. Then $\CEC$ together with the inclusions $C^*(E_X) \into \CEC$ is the amalgamated free product
	$$C^*(E,C)\cong \freeprod{C_0(E^0)}{C^*(E_X)}$$	
of the family of \cstar{}algebras $(C^*(E_X))_{X \in C}$ over the common \cstar{}subalgebra $C_0(E^0)$. 
\end{proposition}

Let us recall (see for instance \cite{Voiculescu-Dykema-Nica:Free, Blackadar:Weak_expectations}) that, in general, the amalgamated free product
$$A=\freeprod{A_0}{A_i}$$	
of a family of \cstar{}algebras $(A_i)_{i \in I}$ with a common \cstar{}subalgebra $A_0\sbe A_i$ 
is the \cstar{}algebra $A$ generated by (copies of) $A_i$ as \cstar{}algebras satisfying the following universal property: Given a \cstar{}algebra $D$ and a family of \Star{}homomorphisms 
$h_i: A_i \to D$ with $h_i|_{A_0}= h_{j}|_{A_0}$ for all $i,j \in I$, there is a unique $h: A\to D$ such that $h_i = h\circ g_i $ for all $i \in I$, where $g_i\colon A_i \to A$ are the canonical $*$-homomorphisms. In the literature most papers only deal with amalgamated free products of unital \cstar{}algebras, but the general case can be covered by taking unitizations. This is already done in \cite{Ara-Goodearl:C-algebras_separated_graphs} for instance.

Proposition~\ref{propcoproduto} provides some examples of \cstar{}algebras of separated graphs. 

\begin{example}\label{exgrafosepararo2}
	Let $(E,C)$ be a separated graph with $E^0=\{v\}$, $E^1=\{e_1,\ldots,e_n,f_1,\ldots,f_m\}$ and $C=C_v:=\{X,Y\}$ with $X=\{e_1,\ldots,e_n\}$ and $Y=\{f_1,\ldots,f_m\}$ as described in Figure \ref{fig:separatedgraph01} below. 
	\begin{center}{
			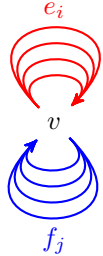
\begin{figure}[htb]
				\begin{tikzpicture}[scale=2]
					\node (v) at (1,0) {$v$};
					\Loop[dist=1cm,dir=NO,label=$e_i$,labelstyle=above,color=red](v)
					\Loop[dist=0.8cm,dir=NO,color=red](v)
					\Loop[dist=0.6cm,dir=NO,color=red](v)
					\Loop[dist=0.4cm,dir=NO,color=red](v)
				 	\Loop[dist=1cm,dir=SO,label=$f_j$,labelstyle=below,color=blue](v)
				 	\Loop[dist=0.8cm,dir=SO,color=blue](v)
				 	\Loop[dist=0.6cm,dir=SO,color=blue](v)
				 	\Loop[dist=0.4cm,dir=SO,color=blue](v)
				\end{tikzpicture}
				\caption{The separated graph with $n+m$ loops}
				\label{fig:separatedgraph01}
		\end{figure}}
	\end{center}

	Consider a group $G$ with generators $g_1,\ldots,g_n,h_1,\ldots,h_m$ and $c:E^1 \to G$ a labeling function defined by $c(e_i):=g_i$ and $c(f_j):=h_j$ for all $i \in \{1,\ldots,n\}$ and $j \in \{1,\ldots,m\}$. In this way we get the skew product separated graph $(E \times_c G,C \times_c G)$. By Proposition \ref{propcoproduto} we have $\CEC \cong C^*(E_X) \star_{\C} C^*(E_Y) \cong \mathcal{O}_n \star_{\C} \mathcal{O}_m$, and consequently, by Theorem \ref{isomorgrafosskewseparproducruzado}, $$\CEGC \cong \CEC \rtimes_{\delta_c} G \cong (\mathcal{O}_n \star_{\C} \mathcal{O}_m)\rtimes_{\delta_c} G$$ where $\delta_c(s_i) = s_i \otimes g_i$ and $\delta_c(t_j) = t_j \otimes h_j$ for every $i,j$. 
\end{example}

\subsection{Conditional expectations and reduced $C^*$-algebras}

In this section we recall the definition of reduced \cstar{}algebras of separated graphs and realize them as certain reduced quotients with respect to canonically constructed conditional expectations. 
The following result from \cite{Ara-Goodearl:C-algebras_separated_graphs} is essential in this context:

\begin{proposition}\cite[Theorem 2.1]{Ara-Goodearl:C-algebras_separated_graphs}\label{esperançacondicional1}
	If $E$ is a row-finite graph, then there is a unique faithful conditional expectation $$\phi_E:C^*(E) \to C_0(E^0)$$ such that, for all paths $\mu, \nu \in \mathrm{Path}(E)$ we have 
	
	$$\phi_E(S_\mu S_\nu^*) = \begin{cases}
		n_\mu P_{s(\mu)}& \text{ if } \mu=\nu \\ 
		0& \text{ if } \mu \neq \nu 
	\end{cases},$$ 
where $n_\mu := \left(\prod_{i=1}^n |s^{-1}(s(\mu_i))|\right)^{-1}$ if $\mu=\mu_1 \ldots \mu_n \in \text{Path}(E)$; and if the length of $\mu$ is zero, we set $n_\mu = 1$.
\end{proposition}

Before we proceed, let us recall some natural conditions involving general conditional expectations that will be used in the sequel:

\begin{definition}\label{defesperançacondicional}
Let $B\sbe A$ be \cstar{}algebras and let $P: A \rightarrow B$ be a conditional expectation. We say that:
	\begin{enumerate}
		\item $P$ is \emph{faithful} if $P(a^*a) = 0$ for some $a \in A$ implies $a=0$,
		\item $P$ is \emph{almost faithful} if $P((ab)^*ab) = 0$ for all $b \in A$ and some $a \in A$ implies $a=0$,
		\item $P$ is \emph{symmetric} if $P(a^*a)=0$ for some $a \in A$ implies $P(aa^*)=0$.
	\end{enumerate}
\end{definition}

\begin{remark}
	Not every almost faithful conditional expectation is faithful. For instance the state of $\Mat_2(\C)$ that reads off the first entry when viewed as a conditional expectation onto $\C\sbe \Mat_2(\C)$  is almost faithful, but not faithful.
\end{remark}

Now, if $(A_i)_{i \in I}$ is a family of \cstar{}algebras containing a common \cstar{}subalgebra $A_0$ with almost faithful conditional expectations  $P_i: A_i \to A_0$, then the \emph{reduced amalgamated free product} of $(A_i)_{i \in I}$ over $A_0$ with respect to $(P_i)_{i\in I}$ is the pair $(A,P)$ consisting of a \cstar{}algebra $A$ containing $A_i$ as a generating family of \cstar{}subalgebras and an almost faithful condition expectation $P\colon A\to A_0$ that restricts to $P_i$ on $A_i$ and such that
\begin{equation}\label{eq:free-condition-exp}
P(a_1\cdot a_2 \cdots a_n) = 0
\end{equation}
whenever $a_k\in \ker(P_{i_k})\sbe A_{i_k}$ and $i_1, \ldots,i_n \in  I$ satisfy $i_1 \neq i_2$, $i_2\not= i_3,\ldots$, $i_{n-1} \neq i_n$.
In this situation, we shall use the notation
$$(A,P)=\rfreeprod{A_0}{(A_i,P_i)}$$
or just $A=\rfreeprod{A_0}{A_i}$ when the conditional expectations involved are implicit. The construction (hence the existence) of reduced amalgamated free products goes back to works by Voiculescu \cite{Voiculescu:Symmetries}, see also \cite{Dykema:Exactness, Voiculescu-Dykema-Nica:Free}. 

Now we are ready to recall the definition of reduced \cstar{}algebras of separated graphs from \cite{Ara-Goodearl:C-algebras_separated_graphs}:

\begin{definition}
Let $(E,C)$ be finitely separated graph $(E,C)$. The reduced \cstar{}algebra of $(E,C)$ is defined as the reduced amalgamated free product
$$\CECR=\rfreeprod{C_0(E^0)}{C^*(E_X)}$$
for $X$ varying over $C$ and each $C^*(E_X)$ is endowed with the canonical conditional expectation from Proposition~\ref{esperançacondicional1}.
\end{definition}

Notice that all the graphs $E_X$ are row-finite by the assumption on finite separation of $(E,C)$. Thus the canonical conditional expectations on $C^*(E_X)$ for $X\in C$ exist by Proposition~\ref{esperançacondicional1}.
And, by definition of the reduced amalgamated free product, $C^*_r(E,C)$ is a \cstar{}algebra containing copies of all the graph \cstar{}algebras $C^*(E_X)$ for $X\in C$ and we have a canonical conditional expectation $\P_r\colon C^*_r(E,C)\to C_0(E^0)$. Using the embeddings $C^*(E_X)\into C^*_r(E,C)$ we view the generating projections $P_v$ and partial isometries $S_e$ of $C^*(E,C)$ also as generators for $C^*_r(E,C)$.
These satisfy the same relations as for $C^*(E,C)$ or for its algebraic companion $L(E,C)$, and hence there is a unique (surjective) *-homomorphism $\Lambda\colon C^*(E,C) \to \CECR$ sending all projections and partial isometries to their canonical images. Moreover, it is shown in \cite[Theorem 3.8]{Ara-Goodearl:C-algebras_separated_graphs} that the restriction of this map to $L(E,C)$ is injective, so we get an embedding $\Lambda\colon L(E,C)\into C^*_r(E,C)$. For ordinary (non-separated graphs) $E$, that is, for the trivial separation $C_v=s^{-1}(v)$ on $E$, we have $C^*(E,C)\cong C^*_r(E,C) \cong C^*(E)$, the usual \cstar{}algebra of $E$, but in general $C^*_r(E,C)$ is a proper quotient of $C^*(E,C)$. 

Returning to a general finitely separated graph $(E,C)$, since all the conditional expectations $C^*(E_X)\to C_0(E^0)$ are faithful, 
it follows from \cite[Theorem 2.1]{Ivanov:structure_amalgamated_free_product} that $\P_r\colon C^*_r(E,C)\to C_0(E^0)$ is faithful as well.

We now introduce a key concept:

\begin{definition}
	\label{def:free-label}
Let $E$ be a directed graph and let $\mathbb F$ be the free group on $E^1$. The
	{\it free label} $\f \colon E^1 \to\mathbb F$ is the map sending each $e\in E^1$ to the canonical generator $\f (e)$ of $\mathbb F$ corresponding to $e$.  
	The free label is universal in the sense that any other labeling $c\colon E^1\to G$ is induced by a unique group homomorphism $c\colon\mathbb F \to G$, which we also denote by $c$ by abuse of notation.
	\end{definition}

In our next result, we use the coaction associated to the free label to show that it can be extended to the $*$-subsemigroup of $L(E,C)\subseteq C^*(E,C)$ generated by $E^0\cup E^1$.

\begin{lemma}
	\label{lem:extending-freelabel}
	Let $(E,C)$ be a separated graph, and consider the $*$-subsemigroup $S(E,C)$
	of $L(E,C)$ generated by $E^0\cup E^1$. Then there is a unique $*$-map $\mathfrak f \colon S(E,C)\setminus \{0\} \to \mathbb F$ such that $\mathfrak f (S_e)$ is the canonical generator of $\mathbb F$ corresponding to $e\in E^1$, and $\mathfrak f (S_\mu \cdot S_\nu) = \mathfrak f (S_\mu)\mathfrak f (S_\nu)$ whenever $S_\mu , S_\nu \in S(E,C)$ and $S_\mu \cdot S_\nu \ne 0$ in $S(E,C)$.
	\end{lemma}

\begin{proof}
	 As observed in Subsection \ref{subsection:skewproductsofgraphs}, we can extend the map  $\f \colon E^1 \to\mathbb F$ to a $*$-functor $ \f \colon \text{Path} (\hat{E})\to \mathbb F$, defined by 
	$$\f (v)= 1, \quad \text{and} \quad \f (a_1a_2\cdots a_r) = 
	\f (a_1)\f (a_2)\cdots \f (a_r),$$
	where $v\in E^0$, $\f (e^*)= \f (e)^{-1}$ for $e\in E^1$, and $a_1,\dots , a_r\in \hat{E}^1$. 
	
	There is a canonical surjective $*$-map $\pi \colon \text{Path} (\hat{E})\to S(E,C)$
sending a path $\mu \in \text{Path}(\hat{E})$ to its canonical image $\pi (\mu):=S_\mu$ in $L(E,C)$.
Note that $\pi(\mu \nu)=\pi(\mu)\pi(\nu)$ whenever $r(\mu)= s(\nu)$. We define $\ol{\f}\colon S(E,C)\setminus \{0\}\to \mathbb F$ by 
$\ol{\f}(\pi (\mu))= \f (\mu)$. We will use in the course of this proof the notation $\ol{\f}$ in order to distinguish the two maps $\ol{\f}$ and $\f$, denoted elsewhere in the paper by the same symbol $\f$.

Let $\mu,\nu \in \text{Path}(\hat{E})$ be such that $\pi(\mu)= \pi (\nu)\ne 0$. We will show that $\f (\mu)= \f (\nu)$. Let $\delta_{\f}\colon C^*(E,C)\to C^*(E,C)\otimes C^*(\mathbb F)$ be the coaction defined in Theorem \ref{isomordeltamaximal}.

By the multiplicative properties of $\delta _{\f}$,  $\pi$ and $\f$, we have $\delta _{\f} (\pi (\mu)) = \pi (\mu)\otimes \f (\mu)$,
and similarly for $\delta_{\f} (\pi (\nu))$. But now, since $\pi(\mu)= \pi (\nu)$, we have
$$\pi (\mu)\otimes \f (\mu) = \delta _{\f} (\pi (\mu)) = \delta _{\f} (\pi (\nu)) = \pi (\nu)\otimes \f (\nu)= \pi (\mu)\otimes \f (\nu).$$
Since $\f (\mu)$ and $\f (\nu)$ are basis elements of $\C [\mathbb F]\subseteq C^*(\mathbb F)$ and $\pi (\mu)\ne 0$, we get $\f (\mu)= \f (\nu)$, as desired.

Hence we have a well-defined $*$-map $\ol{\f}\colon S(E,C)\to \mathbb F$. Suppose that $\mu,\nu\in \text{Path}(\hat{E})$ satisfy that $S_\mu \cdot S_\nu \ne 0$. Then $r(\mu)=s(\nu)$, and 
$$\ol{\f}(S_\mu\cdot S_\nu)= \ol{\f}(S_{\mu\nu})= \ol{\f}(\pi(\mu
\nu))= \f (\mu\nu)= \f(\mu)\f(\nu)= \ol{\f}(S_\mu)\ol{\f}(S_\nu),$$
as desired.  
\end{proof}

We now provide some useful observations about the $*$-semigroup $S(E,C)$. 
The structure of this $*$-semigroup and other related semigroups will be further studied in a forthcoming paper.

 Let $\mu$ be a path in $\hat{E}$. If $\mu$ is not $C$-separated then it can be reduced to a new path
 $\mu'$ such that $|\mu'|= |\mu|-2$, or to $0$, by replacing a subword $e^*f$, with $e,f\in X\in C$, with $\delta_{e,f}r(e)$. Note that either $S_{\mu'}=S_\mu$ or $S_\mu=0$. Continuing in this way (in case that $\mu'\ne 0$), we see that either $S_\mu=0$ or it can be reduced to a $C$-separated path $\nu$ such that $S_{\mu}=S_{\nu}$. 
 If $S_{\nu}$ is not weakly reduced, then it can be reduced by replacing a subword of the form $ee^*$, with $\{ e\}\in C$, with $s(e)$, to another path $\nu'$ such that $|\nu'|=|\nu|-2$ and $S_{\nu'}=S_{\nu}$. It might be that $\nu'$ is not $C$-separated, but in any case we can apply the above process to it in order to arrive ,  after finitely many steps,  to either a $C$-separated weakly reduced path $\zeta$ such that $S_{\zeta}=S_\mu$, or to $0$. Of course, the latter case occurs if and only if $S_\mu=0$. Note also that $s(\zeta) = s(\mu)$ and $r(\zeta)= r(\mu)$.

Write $\mathcal C:= \text{Path}(\hat{E})$, and $\mathcal C_0:=\mathcal C\cup \{0\}$. The reductions in $\mathcal C$ of the form $\mu_0e^*e\mu_1\mapsto \mu_0\mu_1$ and $\mu_0ee^*\mu_1\mapsto \mu_0\mu_1$, for $e\in E^1$, will be called free-group reductions. It is clear that if $\mu'$ is obtained from $\mu$ by applying a finite number of free-group reductions, then $\f(\mu) = \f(\mu')$. 
We will write $\mu \equiv \mu'$ if $\f(\mu) = \f(\mu')$, for $\mu,\mu'\in \mathcal C$. As observed before, if $\mu\in \mathcal C$, then either it can be reduced to $0$ by using the reduction rules $e^*f\mapsto \delta_{e,f} r(e)$ if $e,f\in X$, for $X\in C$, and $ee^*\mapsto s(e)$ if $\{e\}\in C$ (and hence $S_\mu =0$), or it can be reduced using solely free-group reductions of the form $e^*e\mapsto r(e)$ for $e\in E^1$ and $ee^*\mapsto s(e)$ for $e\in E^1$ such that $\{ e\}\in C$, to a $C$-separated weakly reduced path $\mu'$, so that in particular $S_\mu = S_{\mu'}$ and $\mu \equiv \mu'$.

\begin{lemma}
	\label{lem:formula-for-P}
	Let $(E,C)$ be a finitely separated graph. Then we have the formula \begin{equation}
	\label{eq:condexpeconbasiselements}
	P(S_{\mu_1} S_{\nu_1}^*\ldots S_{\mu_n} S_{\nu_n}^*) =	N_{\mu}P_{s(\mu)}
	\end{equation} 
	for each $C$-separated weakly reduced path $\mu\defeq \mu_1\nu_1^* \cdots \mu_n \nu_n^*$, where $N_{\mu}$ is a rational number. Moreover $P(S_{\mu_1} S_{\nu_1}^*\ldots S_{\mu_n} S_{\nu_n}^*)=0$ whenever $\f (\mu)\ne 1$, where $\f (\mu)$ is the canonical image of $\mu$ in the free group $\mathbb F$ generated by $E^1$. 	
\end{lemma}

\begin{proof}
	For $X\in C_v$ such that $|X|>1$ and $e\in X$, we write
	$$\beta _e= S_eS_e^*- \frac{1}{|X|}\cdot P_v .$$
	Note that $S_\lambda \beta _e S_\tau^* \in \ker (P_X)$ for each $e\in X$ and each pair of paths $\lambda ,\tau$ in $E_X$ such that $r(\lambda)= r(\tau)= v$.   
	
	Let $\mu $ be a $C$-separated weakly reduced path.
	We first deal with the case where $\mu = \mu _1\nu_1^*$, where $\mu_1$ and $\nu_1$ are paths in $E$. We can express $\mu$ in one of the following forms:
	\begin{enumerate}
		\item Either 	$$\mu = \lambda _1\cdots \lambda_j\lambda _{j+1}^* \cdots \lambda_{j+r}^*,$$
		where $j,r\ge 0$, $\lambda _i$ are paths of positive length in $E_{X_{l_i}}$ for $i=1,\dots , j+r$, and $X_{l_{i}}\ne X_{l_{i+1}}$ for $i=1,\dots , j+r-1$. If $j=r=0$, we interpret this path as $\mu = v$,
		\item or
		$$\mu = \lambda _1\cdots \lambda_j\cdot \delta \cdot \lambda _{j+1}^* \cdots \lambda_{j+r}^*,$$
		where $j,r\ge 0$, $\lambda _i$ are paths of positive length in $E_{X_{l_i}}$ for $i=1,\dots , j+r$, $\delta = \lambda ef^* \tau^* $, where $v:=s(\lambda) =r(\lambda )= s(e)=s(f)= r(\tau)= s(\tau)$, and $\lambda e$, $\tau f$ are paths in $E_X$ for some $X\in C_v$. Moreover we have $X_{l_{i}}\ne X_{l_{i+1}}$ for $i=1,\dots , j-1$, $X_{l_j}\ne X$, $X\ne X_{l_{j+1}}$, and $X_{l_{j+i}}\ne X_{l_{j+i+1}}$ for $i=1,\dots , r-1$. 
	\end{enumerate}
	In case (1), we have  $P(S_\mu )= P(S_{\lambda _1}\cdots S_{\lambda_j} \cdot S_{\lambda _{j+1}}^* \cdots S_{\lambda_{j+r}}^*)=0$ by \eqref{eq:free-condition-exp}, except in the case where $\mu = v\in E^0$. 
	
	In case (2), we distinguish two subcases:
	
	(2a) $e\ne f$.  In this case, $P_X(S_\delta)= 0$ and so $P(S_\mu)= 0$ by \eqref{eq:free-condition-exp}. 
	
	(2b) $e=f$. In this case, we can write 
	$$S_\mu = \frac{1}{|X|}S_{\lambda _1}\cdots S_{\lambda_j}S_{\lambda} S_{\tau}^* S_{\lambda_{j+1}}^* \cdots S_{\lambda_{j+r}}^*+ S_{\lambda_1}\cdots S_{\lambda_j}(S_{\lambda} \beta_e S_{\tau}^*)S_{\lambda_{j+1}}^* \cdots S_{\lambda_{j+r}}^*,$$
	where $e\in X\in C$. 
	We have $P(S_{\lambda_1}\cdots S_{\lambda_j}(S_{\lambda} \beta_e S_{\tau}^*)S_{\lambda_{j+1}}^* \cdots S_{\lambda_{j+r}}^*)=0$ by \eqref{eq:free-condition-exp}, and thus after reducing $\gamma := \lambda _1\cdots \lambda_j \lambda \tau^* \lambda_{j+1}^* \cdots \lambda_{j+r}^*$ we get a $C$-separated weakly reduced path $\gamma'$ (or $0$), such that $S_{\gamma}=S_{\gamma'}$ and $\gamma' \equiv \gamma \equiv \mu$. Applying induction on the length of $\mu$ (as a $C$-separated weakly reduced path), we deduce that $P(S_\mu)=P(S_{\gamma'})$ is of the form
	$N_\mu P_{s(\mu)}$, where $N_\mu$ is a rational number.  
	
	Furthermore, suppose that $\f (\mu) \ne 1$ as an element of $\mathbb F$, the free group on $E^1$. Then, since $\mu \equiv \gamma'$, it follows by induction that $P(S_{\gamma'})=0$, and thus $P(S_\mu)= P(S_{\gamma'})=0$. 
	
	We now consider the general case where $\mu = \mu_1\nu_1^*\cdots \mu_n\nu_n^*$ is a $C$-separated weakly reduced path. If all the 
	paths $\mu_i \nu_i^*$ are in cases (1) or (2a) above, then we get
	$P(S_\mu)=0$. For those $\mu_i\nu_i^*$ which are in case (2b) above, we can 
	replace $S_{\mu_i} S_{\nu_i}^*$ by $(1/|X_i|)S_{\mu_i'}(S_{\nu_i'})^*+S_{\mu_i'}\beta_{e_i}(S_{\nu_i'})^*$, where $\mu_i=\mu_i'e_i$ and $\nu_i= \nu_i'e_i$, and $e_i\in X_i\in C$.
	
	Let $i_1,i_2,\dots , i_t$ be the set of indices $i$ for which $\mu_i\nu_i^*$ is of the form (2b). When we perform all the substitutions mentioned above, and we expand them as a sum of products, all of the terms but one will involve a factor of the form $S_{\mu_i'} (S_{\nu_i'})^*$, which has length strictly less than the length of $S_{\mu_i} S_{\nu_i}^*$. Therefore, we can write 
	\begin{align*}
		S_\mu & = (S_{\mu_1}S_{\nu_1}^*)\cdots (S_{\mu_{i_1}'}\beta_{e_i}(S_{\nu_{i_1}'})^*)\cdots  (S_{\mu_{i_2}'}\beta_{e_2}(S_{\nu_{i_2}'})^*)\cdots (S_{\mu_{i_t}'}\beta_{e_t}(S_{\nu_{i_t}'})^*)\cdots (S_{\mu_n}S_{\nu_n}^*)\\ &  +  \sum_{p=1}^M d_pS_{\delta_p},
	\end{align*}
	where $\delta_p$ are paths in $\hat{E}$ of length strictly smaller than the length of $\mu$, and $d_p$ are rational numbers. Now, for a given $p$, either $S_{\delta_p}=0$, or $\delta _p$ can be reduced using free-group reductions to a $C$-separated weakly reduced path $\delta_p'$, representing the same element as $\delta_p$ in $C^*(E,C)$, that is, $S_{\delta_p}=S_{\delta_p'}$,  and in the latter case we have $\delta _p' \equiv \delta _p \equiv \mu$, so that $\mathfrak f (\mu)= \mathfrak f (\delta_p')$.   
	
	Since 
	$$P((S_{\mu_1}S_{\nu_1}^*)\cdots (S_{\mu_{i_1}'}\beta_{e_i}(S_{\nu_{i_1}'})^*)\cdots  (S_{\mu_{i_2}'}\beta_{e_2}(S_{\nu_{i_2}'})^*)\cdots (S_{\mu_{i_t}'}\beta_{e_t}(S_{\nu_{i_t}'})^*)\cdots (S_{\mu_n}S_{\nu_n}^*) ) =0,$$
	we can apply induction to get our result. Moreover, if $\mathfrak f(\mu)\ne 1$, then we get by induction that $P(S_{\delta_p})=P(S_{\delta'_p})=0$ for all $p$, and so $P(S_\mu)=0$. 
\end{proof}

For later use, we also need to study the invariance of $P$ with respect to a group action, as follows.

\begin{lemma}
	\label{lem:formula-for-P-groupaction}
	Let $(E,C)$ be a finitely separated graph and let $\alpha $ be an action of a group $G$ on $(E,C)$. Then we have $N_{\alpha_g(\mu)}=N_\mu$ for each $C$-separated weakly reduced path $\mu =\mu_1\nu_1^* \cdots \mu_n \nu_n^*$.
	\end{lemma}

\begin{proof} First of all, note that the set of $C$-separated weakly reduced paths is invariant under the action of $G$.

	 To simplify the notation,
	we will write $g\cdot \mu$ instead of $\alpha_g(\mu)$. We will use the same notation as in Lemma \ref{lem:formula-for-P}, and also we will follow the same steps in order to show the result.
	
Assume first that $\mu = \mu _1\nu_1^*$, where $\mu_1$ and $\nu_1$ are paths in $E$. If $\mu$ is of the form (1) of Lemma \ref{lem:formula-for-P}, then 
$g\cdot \mu$ is also in the form (1), and thus $P(S_{g\cdot \mu})= P(S_\mu)= 0$ except that $\mu = v$ for $v\in E^0$, in which case we have $N_{g\cdot v} = 1 = N_v$. If $\mu$ falls in case (2a) of Lemma \ref{lem:formula-for-P}, then also $g\cdot \mu$ falls in case 2(a), and thus $N_{g\cdot \mu}= 0 = N_{\mu}$.  
Suppose that $\mu$ is in the form (2b) of Lemma \ref{lem:formula-for-P}.  
In this case, using the same notation as in that lemma, we can write 
\begin{align*}
S_{g\cdot \mu} & = \frac{1}{|X|}S_{g\cdot \lambda _1}\cdots S_{g\cdot \lambda_j} S_{g\cdot \lambda} S_{g\cdot \tau}^* S_{g\cdot \lambda_{j+1}}^* \cdots S_{g\cdot \lambda_{j+r}}^* \\ & + S_{g\cdot \lambda_1}\cdots S_{g\cdot \lambda_j}(S_{g\cdot \lambda} \beta_{g\cdot e} S_{g\cdot \tau}^*)S_{g\cdot \lambda_{j+1}}^* \cdots S_{g\cdot \lambda_{j+r}}^*,
\end{align*}
where $e\in X\in C$. 
We have $$P(S_{g \cdot \lambda_1}\cdots S_{g \cdot \lambda_j} (S_{g \cdot \lambda} \beta_{g\cdot e} S_{g \cdot \tau}^*)S_{g \cdot  \lambda_{j+1}}^* \cdots S_{g \cdot \lambda_{j+r}}^*)=0,$$ and it is a simple matter to check that the reduction of a path $g\cdot \tau$ is of the form $g\cdot \tau'$, where $\tau'$ is the reduction of $\tau$. Hence using induction on the length of $\mu$ (as a $C$-separated weakly reduced path), we deduce that $N_{g\cdot \mu}= N_\mu$.   

The general case of a $C$-separated weakly reduced path 
$\mu = \mu_1\nu_1^*\cdots \mu_n\nu_n^*$ uses a similar reasoning, with the same reduction steps from Lemma \ref{lem:formula-for-P}. We leave the details to the reader. 
\end{proof}

 \subsection{The reduced $C^*$-algebra associated with a conditional expectation}

The original definition of the reduced \cstar{}algebra $C^*_r(E,C)$ given in the previous section uses the machinery of reduced amalgamated free products.
It produces a special quotient of the full \cstar{}algebra $C^*(E,C)$ carrying a faithful conditional expectation $\P_r\colon C^*_r(E,C)\to C_0(E^0)$.  
We can then pullback the conditional expectation from $C^*_r(E,C)$ to $C^*(E,C)$ via the quotient homomorphism $\Lambda\colon C^*(E,C)\onto C^*_r(E,C)$ in order to get also a conditional expectation 
$$\P=\P_r\circ\Lambda:C^*(E,C) \rightarrow C_0(E^0).$$
We also refer to this as the canonical conditional expectation of $C^*(E,C)$. Of course, it is given by the same formula~\eqref{eq:condexpeconbasiselements} on the generators.

In this section we are going to show how to build $C^*_r(E,C)$ from $C^*(E,C)$ and the canonical conditional expectation $\P$. 

Indeed, this works in general for abstract \cstar{}algebras and conditional expectations:

\begin{proposition}[\cite{Kwasniewski-Meyer:Essential}*{Section~3.1}]\label{prop:KM-reduced-quotient-exp}
Let $A$ be a \cstar{}algebra and let $P\colon A\to B$ be a conditional expectation onto a \cstar{}subalgebra $B\sbe A$.
Let $\NN_P$ be the largest (closed, two-sided) ideal contained in $\ker(P)$, that is, the (closed) sum of all closed ideals of $A$ contained in $\ker(P)$; we call $\NN_P$ the \emph{nucleus} of $P$.
Then the quotient \cstar{}algebra $A_{P,r}\defeq A/\NN_{P}$ contains (a copy of) $B$ as a \cstar{}subalgebra and $P$ factors through an almost faithful conditional expectation $P_r\colon A_{P,r}\to B$. Moreover, 
$$\mathcal{N}_P = \{ a \in A ~|~ P((ab)^*(ab))=0 \text{ for all } b \in A\},$$
and $\mathcal{L}_P = \{ a \in A ~|~ P(a^*a)=0\} \text{ and }\mathcal{R}_P = \{ a \in A ~|~ P(aa^*)=0\}$ are the largest left and right ideals in $A$ contained in $\ker(P)$, respectively. We have
	\begin{enumerate}
		\item $P$ is symmetric if and only if $\mathcal{L}_P=\mathcal{R}_P=\mathcal{N}_P$,
		\item $P$ is faithful if and only if $\mathcal{L}_P=\mathcal{R}_P=\mathcal{N}_P=0$,
		\item $P$ is almost faithful if and only if $\mathcal{N}_P=0$,
		\item $P$ is faithful if and only if $P$ is almost faithful and symmetric.
	\end{enumerate}
\end{proposition}

It follows that $A_{P,r}$ is the unique quotient of $A$ containing $B$ as a \cstar{}subalgebra such that $P$ factors through an almost faithful conditional expectation onto $B$, more precisely: 

\begin{corollary}\label{cor:prop-KM-red-quotient-exp}
If  $\Lambda\colon A\onto A_{P,r}$ denotes the quotient map and if $D$ is another \cstar{}algebra containing $B$ with a surjective *-homomorphism $\pi: A \twoheadrightarrow D$ and an almost faithful conditional expectation $Q: D \to B$ which factors $P$, then there is a unique isomorphism $\psi \colon D\congto A_{P,r}$ such that $\psi \circ\pi= \Lambda$:
	\[
	\begin{tikzcd}
		A \arrow[twoheadrightarrow]{r}{\pi} \arrow[swap]{d}{P} & D \arrow[swap]{dl}{Q}\arrow{r}{\psi} & A_{P,r} \arrow[swap]{dll}{P_r}\\
		B  & 
	\end{tikzcd}
	\] 
\end{corollary}

The reduced \cstar{}algebras $A_{P,r}$ usually realize some sort of ``reduced crossed product''. Needless to say, the usual reduced crossed products $B\rtimes_\red G$ for an action of a discrete group $G$ on a \cstar{}algebra $B$ are special examples of the above construction by taking, for instance, the full crossed product $B\rtimes G$ as the starting \cstar{}algebra $A$ endowed with the canonical conditional expectation $B\rtimes G\onto B$.
More generally, one can consider also Fell bundles and their cross-sectional \cstar{}algebras as special examples. 

For us, the most important class of examples are those associated to separated graphs where we have an ordinary conditional expectation taking values into a commutative subalgebra:

\begin{corollary}\label{cor:red-separated-cond-exp}
Let $(E,C)$ be a finitely separated graph and let $P\colon C^*(E,C)\to C_0(E^0)$ be the canonical conditional expectation. This is symmetric and its reduced quotient is
$$C^*(E,C)_{P,r}\cong C^*_r(E,C)$$
endowed with the canonical faithful conditional expectation given by~\eqref{eq:condexpeconbasiselements}.
\end{corollary}

\begin{proof}
	This follows from Corollary \ref{cor:prop-KM-red-quotient-exp} and the fact that the canonical conditional expectation on $C^*_r(E,C)$ is faithful. 
\end{proof}

The algebras $C^*(E,C)$ and $C^*_r(E,C)$ are, in general, quite different as is illustrated by the following basic example.

\begin{example}\label{ex:red-Cuntz}
Let $(A_n,D)$ be the Cuntz separated graph from Example \ref{excuntzgrafoseparado}. 
As observed in Example~\ref{excanonicografoseparado}, we have $C^*(A_n,D)=C^*(\Free)$,  the full \cstar{}algebra of the free group $\mathbb{F}$ on the edges $A_n^1$ of $A_n$. The canonical conditional expectation coincides with the canonical trace $\tau$ on $C^*(\Free)$,
so that $C^*_r(A_n,D)=C^*(\F)_{\tau,r} \cong C_r(\F)$. 
\end{example}

Of course, $C^*(E,C)=C^*_r(E,C)$ if and only if the canonical conditional expectation $P$ is faithful on $C^*(E,C)$. This happens, for instance, for ordinary graph \cstar{}algebras, but also for certain special properly separated graphs as shown in the following example, which is a particular case of \cite[Proposition 9.1]{Ara-Exel:Dynamical_systems}.  

\begin{example} 
	Consider the separated graphs $(E(m,n),C(m,n))$ from Example~\ref{exemplografoseparado}.
For $m=1$ and $n\ge 2$, we have the separated graph $(E(1,n),C(1,n))$, shown in Figure \ref{fig:separtedgraphE1n}. 
\begin{center}{
			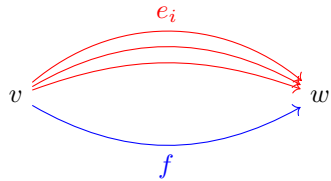
\begin{figure}[htb]
				\begin{tikzpicture}[scale=2]
					\node (v) at (0,0) {$v$};
					\node (w) at (2,0) {$w$};
					\draw[->,red]  (v) to [bend left=40] node [above] {$e_i$} (w);
					\draw[->,red]  (v) to [bend left=30] (w);
					\draw[->,red]  (v) to [bend left=20] (w);
					\draw[->,blue]  (v) to [bend right] node [below] {$f$} (w);
				\end{tikzpicture}
				\caption{The separated graph $(E(1,n),C(1,n))$}
				\label{fig:separtedgraphE1n}
		\end{figure}}
\end{center}
\vspace{0.2cm}
We know from Example~\ref{excanonicografoseparado1} that $C^*(E(1,n),C(1,n)) \cong M_2(\mathcal{O}_n)$.
Identifying $C_0(E^0)=\C \oplus \C$ as diagonal matrices in $M_2(\mathcal{O}_n)$, the canonical conditional expectation is given by
$$P: M_2(\mathcal{O}_n) \to \C\oplus\C,\quad P(a)=(\tau(a_{11}),\tau(a_{22}))$$ 
for $a=(a_{ij})\in M_2(\mathcal{O}_n)$, where $\tau$ is the canonical faithful state on $\mathcal{O}_n$ (consider $\mathcal{O}_n$ as a graph \cstar{}algebra and apply Proposition~\ref{esperançacondicional1}). Therefore $P$ is faithful and we have 
	$$C_r^*(E(1,n),C(1,n)) = C^*(E(1,n),C(1,n)).$$

For a separated graph $(E(m,n),C(m,n))$ with $1< m<n$, the behaviour is quite different.
It is shown in \cite[Corollary 4.5]{Ara:Purely_infinite} that the reduced \cstar{}algebra  $C^*_r(E(m,n),C(m,n))$ is a simple purely infinite \cstar{}algebra. 
On the other hand, $C^*(E(m,n),C(m,n))$ is not a simple \cstar{}algebra, so that $C^*(E,C)\not=C^*_r(E,C)$.
\end{example}

From now on, we assume that all conditional expectations are symmetric. 

\begin{lemma}\label{lemaesperançacondicionalAPR}
	Let $A$ and $A'$ be \cstar{}algebras with symmetric conditional expectations $P:A \rightarrow B$ and $Q: A' \rightarrow B'$ and suppose that $\pi: A \rightarrow A'$ is a *-homomorphism commuting with these conditional expectations, that is, the diagram below commutes: 
	\[
	\begin{tikzcd}
		A \arrow[rr]{}{\pi} \arrow[d]{}{P} & & A' \arrow[d]{}{Q} \\
		B \arrow[rr]{}{\pi|} & & B'
	\end{tikzcd}
	\] Then $\pi$ factors through a *-homomorphism $\pi_r:A_{P,r} \rightarrow A'_{Q,r}$. Moreover, if $\pi$ is surjective, then so is $\pi_r$, and if $\pi$ is injective on $B$, then $\pi_r$ is injective on $A_{P,r}$.
	In particular, $\pi_r$ is an isomorphism if so is $\pi$.
\end{lemma}

\begin{proof}
	Let $\Lambda: A \to A_{P,r}$ and $\Lambda':A' \to A'_{Q,r}$ denote the quotient maps. We want to define $\pi_r: A_{P,r} \to A'_{Q,r}$ such that $\pi_r(\Lambda(a))=\Lambda'(\pi(a))$ for all $a \in A$. To see that such map exists, we must show that $\Lambda(a)=0$ implies $\Lambda'(\pi(a))=0$, that is, $\pi(\ker(\Lambda))\sbe \ker(\Lambda')$. But since $P$ and $Q$ are assumed to be symmetric, $\ker(\Lambda)=\LL_P$ and $\ker(\Lambda')=\LL_Q$, and the necessary inclusion is equivalent to $\pi(\LL_P)\sbe \LL_Q$, which follows from $Q\circ \pi=\pi\circ P$. Hence $\pi_r$ exists. It is clear from the definition that $\pi_r$ is surjective if so is $\pi$. Suppose that $\pi$ is injective on $B$ and let $a\in A$ with $\pi_r(\Lambda(a))=\Lambda'(\pi(a)) = 0$. We must show that $\Lambda(a) = 0$. But 
$$\pi(P(a^*a))=Q(\pi(a^*a))=Q_r(\Lambda'(\pi(a^*a)))=Q_r(\Lambda'(\pi(a))^*\Lambda'(\pi(a)))=0.$$
Since $\pi$ is injective on $B$, it follows that $P(a^*a)=0$, that is, $a\in \LL_P=\NN_P=\ker(\Lambda)$.
\end{proof}

Next we are going to record some results that will be needed later. They describe the reduced quotients by conditional expectations in some special situations.

\begin{lemma}\label{lemaPotimesQ}
	Let $A$ and $A'$ be \cstar{}algebras with symmetric conditional expectations $P:A \rightarrow B$ and $Q: A' \rightarrow B'$. 
	Consider on the (minimal) tensor product $A\otimes A'$ the tensor product conditional expectation $P\otimes Q\colon A\otimes A'\to B\otimes B'$.
	Then $$(A \otimes A')_{P \otimes Q,r} \cong A_{P,r} \otimes A'_{Q,r}.$$
\end{lemma}

\begin{proof}
First of all, it is well known that $P \otimes Q : A \otimes A' \to B \otimes B'$ is a well-defined conditional expectation given on generators by $(P \otimes Q)(a \otimes a') = P(a)\otimes Q(a')$, see \cite{Brown-Ozawa:Approximations}*{Theorem~3.5.3} or \cite{Blackadar:book}*{II.9.7.1}. By Proposition~\ref{prop:KM-reduced-quotient-exp} it is enough to check that $A_{P,r} \otimes B_{Q,r}$ has a faithful conditional expectation which factors $P \otimes Q$. But the tensor product $P_r \otimes Q_r: A_{P,r} \otimes A'_{Q,r} \to B \otimes B'$ is a conditional expectation satisfying these properties.	
\end{proof}

\begin{proposition}\label{induçãoesperançacondicional}
	Let $(A,G,\delta)$ be a coaction of $G$, let $B\sbe A$ be a $G$-invariant \cstar{}subalgebra in the sense that $\delta$ restricts to a coaction $\delta\colon B\to B\otimes C^*(G)$ and let $P: A \rightarrow B$ a $G$-equivariant symmetric conditional expectation meaning that $(P \otimes\id_G) \circ \delta = \delta \circ P$. Then there is a conditional expectation $P \rtimes G: A \rtimes_\delta G \rightarrow B \rtimes_\delta G$ such that 
	$$P \rtimes G(j_A(a)j_G^A(f)) = (j_B \circ P)(a)j_G^B(f)$$ for every $a \in A$ and $f \in C_0(G)$. Moreover, if $P$ is faithful then so is $P \rtimes G$.
\end{proposition}

\begin{proof}
Recall that the crossed product $A\rtimes_\delta G$ is a \cstar{}subalgebra of $\M(A\otimes\K(\ell^2G))$ by definition.
It will be important for the proof to recall from \cite{Echterhoff-Kaliszewski-Quigg-Raeburn:Categorical}*{Remark A.40} that it actually always lies in the $\K$-multiplier \cstar{}subalgebra:
$$\M_\K(A\otimes\K(\ell^2G))\defeq \{x\in\M(A\otimes\K): x(1\otimes k), (1\otimes k)x\in A\otimes \K(\ell^2G), \forall k\in \K(\ell^2G)\}.$$
Since $\M_{\K}(B \otimes \K(\ell^2 G)) \subseteq \M_{\K}(A \otimes \K(\ell^2 G))$ (see \cite{Echterhoff-Kaliszewski-Quigg-Raeburn:Categorical}*{Proposition A.6}), it follows that we can view $B\rtimes_\delta G$ as a \cstar{}subalgebra of $A\rtimes_\delta G$ via the regular representations. Moreover, the conditional expectation $P\colon A\to B$ extends to a conditional expectation $P\otimes\Id_\K\colon \M_\K(A\otimes\K(\ell^2G))\to \M_\K(B\otimes\K(\ell^2G))$ and since $P$ commutes with the coactions on $A$ and $B$, it follows that $P\otimes\Id_\K$ restricts to a conditional expectation $P\rtimes G\colon A\rtimes_\delta G\to B\rtimes_\delta G$ as in the statement. If $P$ is faithful then so is $P \otimes\id_\K$ and therefore also $P\rtimes G$. 
\end{proof}

\begin{lemma}\label{lemaAprotmesG}
	Let $\delta\colon A\to A\otimes C^*(G)$ be a coaction and let $P: A \to B$ be a $G$-equivariant symmetric conditional expectation onto a $G$-invariant \cstar{}subalgebra $B\sbe A$. Then  
	$\delta$ factors through a coaction $\delta_r\colon A_{P,r}\to A_{P,r}\otimes C^*(G)$ such that 
\begin{equation}\label{eq:coaction-Pr}
(A \rtimes_{\delta} G)_{P \rtimes G,r} \cong A_{P,r} \rtimes_{\delta_r} G.
\end{equation}
Moreover, if $(P\otimes\tau)\circ\delta=P$, where $\tau\colon C^*(G)\to \C$ denotes the canonical trace, then the coaction $\delta_r$ is normal.
\end{lemma}

\begin{proof} Since $P$ is $G$-equivariant, combining Lemmas~\ref{lemaesperançacondicionalAPR} and~\ref{lemaPotimesQ}, we get a \Star{}homo\-mor\-phism 
$$\delta_r\colon A_{P,r}\to (A\otimes C^*(G))_{P\otimes\Id}=A_{P,r}\otimes C^*(G)$$ 
satisfying $\delta_r (\Lambda (a))= (\Lambda \otimes\id_G)(\delta (a))$ for all $a\in A$. The coaction identity $(\delta_r\otimes\Id)\circ\delta_r=(\Id\otimes\delta_G)\circ\delta_r$ easily follows from this relation.
And similarly the nondegeneracy of $\delta_r$ also follows:
\begin{multline*}
\cspn \delta_r(A_{P,r})(1\otimes C^*(G))=\cspn (\Lambda\otimes\Id)(\delta(A)(1\otimes C^*(G)))\\
=\cspn (\Lambda\otimes \Id)(A\otimes C^*(G))=A_{P,r}\otimes C^*(G).
\end{multline*}	
	
	Let $P_r: A_{P,r} \to B$ be the faithful conditional expectation that factors $P$. By Proposition \ref{induçãoesperançacondicional} we have conditional expectations $P \rtimes G: A \rtimes_{\delta} G \to B \rtimes_{\delta} G$ and $P_r \rtimes G: A_{P,r} \rtimes_{\delta_r} G \to B \rtimes_{\delta} G$. Since $P_r$ is faithful then so is $P_r \rtimes G$ and the diagram 
	\[	
	\begin{tikzcd}
		A \rtimes_{\delta} G \arrow{r}{\Lambda \rtimes G} \arrow{d}{P \rtimes G} & A_{P,r} \rtimes_{\delta_r} G \arrow{ld}{P_r \rtimes G} \\
		B \rtimes_{\delta} G   & 
	\end{tikzcd}
	\] 
commutes because 
$(P_r\rtimes G)\circ (\Lambda\rtimes G)=(P_r\circ\Lambda)\rtimes G=P\rtimes G$. The isomorphism~\eqref{eq:coaction-Pr} then follows from Proposition~\ref{prop:KM-reduced-quotient-exp}.

To prove the normality of $\delta_r$ under the assumption that $(P\otimes\tau)\circ\delta=P$ we apply a similar idea, but now with the conditional expectation $P\otimes\tau\colon A\otimes C^*(G)\to B$.
Notice that $C^*(G)_{\tau,r}=C^*_r(G)$ with the regular representation $\lambda\colon C^*(G)\to C^*_r(G)$ as the quotient map, and the induced trace $\tau_r$ on $C^*_r(G)$ is faithful. It follows from Lemmas~\ref{lemaesperançacondicionalAPR} and~\ref{lemaPotimesQ} that there is a \Star{}homomorphism
$$\delta_r^\lambda\colon A_{P,r}\to (A\otimes C^*(G))_{P\otimes\tau,r}=A_{P,r}\otimes C^*_r(G)$$
such that $\delta_r^\lambda\circ\Lambda=(\Lambda\otimes\lambda)\circ\delta=(\Id\otimes\lambda)\circ(\Lambda\otimes\Id)\circ\delta=(\Id\otimes\lambda)\circ\delta_r\circ\Lambda$, that is, $\delta_r^\lambda=(\Id\otimes\lambda)\circ\delta_r$. By construction and the hypothesis that $(P\otimes \tau)\circ \delta = P$, we also have $(P_r\otimes\tau_r)\circ\delta_r^\lambda=P_r$. Since $P_r$ is faithful, it follows that $\delta_r^\lambda$ is also faithful.
But this means that $\delta_r$ is normal, as desired.
\end{proof}

\subsection{Duality for reduced separated graph $C^*$-algebras}

In this section we are going to use the description of the reduced \cstar{}algebras $C^*_r(E,C)$ of finitely separated graphs $(E,C)$ in terms of the reduced quotients described in the previous section, specially in Corollary~\ref{cor:red-separated-cond-exp}, in order to derive results involving coactions and their crossed products associated to labels on $E$, similarly to what has been done for the full \cstar{}algebras $C^*(E,C)$ in Section~\ref{sec:duality-full-separated}.

Let $(E,C)$ be a finitely separated graph and let $c: E^1 \to G$ be a labeling function. We already have from Theorem~\ref{isomordeltamaximal} a coaction $\delta_c: C^*(E,C) \rightarrow C^*(E,C) \otimes C^*(G)$ and we next prove that this coaction factors through the reduced \cstar{}algebra:

\begin{proposition}\label{prop:coaction-reduced-separated}
	Let $(E,C)$ be a finitely separated graph and $c: E^1 \to G$ be a labeling function. Then the coaction $\delta_c:\CEC \to \CEC \otimes \CG$ factors through the reduced \cstar{}algebra $C_r^*(E,C)$, that is, there is a coaction $\delta_c^r: \CECR \to \CECR \otimes \CG$ such that $\delta_c^r\circ \Lambda = (\Lambda \otimes\id_G)\circ \delta_c $. Moreover, $\delta_c^r$ is a normal coaction.  
\end{proposition}
\begin{proof}
We apply Lemma~\ref{lemaAprotmesG} with $A=C^*(E,C)$, $B=C_0(E^0)$ and the canonical conditional expectation $P\colon A\to B$. 
Recall that the coaction $\delta_c$ is trivial on $B$, so that $B$ is $G$-invariant. Therefore Lemma~\ref{lemaAprotmesG} already gives us the desired coaction $\delta_c^r$. And by the same lemma, to show that $\delta_c^r$ is normal, it is enough to prove that
$(P\otimes\tau)\circ\delta_c=P$, where $\tau$ is the canonical trace of $\CG$. 
Consider a basic element of the form $S_{\mu_1} S_{\nu_1}^* \cdots S_{\mu_n} S_{\nu_n}^* \in L(E,C)$ for a $C$-separated reduced path $\mu\defeq \mu_1\nu_1^* \cdots \mu_n \nu_n^*$ in $(E,C)$. 
Using \eqref{eq:condexpeconbasiselements}, we get
	\begin{align*}
		(P \otimes \tau) \circ \delta_c(S_{\mu_1} S_{\nu_1}^* & \cdots S_{\mu_n} S_{\nu_n}^*) \\
		&=(P \otimes \tau)(S_{\mu_1} S_{\nu_1}^* \cdots S_{\mu_n} S_{\nu_n}^* \otimes c(\mu_1)c(\nu_1)^{-1} \ldots c(\mu_n)c(\nu_n)^{-1})\\
		&=P(S_{\mu_1} S_{\nu_1}^* \cdots S_{\mu_n} S_{\nu_n}^*) \tau(c(\mu_1)c(\nu_1)^{-1} \cdots c(\mu_n)c(\nu_n)^{-1}) \\
		&=N_\mu P_{s(\mu)}=P(S_{\mu_1} S_{\nu_1}^* \cdots S_{\mu_n} S_{\nu_n}^*).
	\end{align*} 
	We need to justify the equality $P(S_{\mu})\tau(c(\mu))= N_\mu P_{s(\mu)}$ above. Recall that we have defined the free label $\f \colon E^1\to \mathbb F$, where $\mathbb F$ is the free group on $E^1$, and that there is a unique group homomorphism $c\colon \mathbb F\to G$ so that $c(\f(e))= c(e)$ for all $e\in E^1$. Suppose first that $c(\mu)=1$. Then $\tau(c(\mu))= 1$ and the equality is clear. If $c(\mu)\ne 1$, then $\f (\mu)\ne 1$, and therefore by Lemma \ref{lem:formula-for-P}, we have  $P(S_\mu)=0$ and so $N_\mu =0$, and the equality is also clear. 
Therefore $(P\otimes\tau)\circ\delta_c=P$ as the elements  $S_{\mu_1} S_{\nu_1}^* \cdots S_{\mu_n} S_{\nu_n}^*$ generate $C^*(E,C)$ as a Banach space and all maps involved are continuous and linear.
\end{proof}

As an immediate consequence we get:  

\begin{corollary}\label{isomorfismo1}
	Let $(E,C)$ be a finitely separated graph and let $c: E^1 \to G$ be a labeling function. Then $$C_r^*(E,C) \rtimes_{\delta_c^r} G \rtimes_{\widehat{\delta_c^r},r} G \cong \CECR \otimes \K(\ell^2 G).$$
\end{corollary}
\begin{proof}
	This follows from the fact that $\delta_c^r$ is a normal coaction. 
\end{proof}

\begin{lemma}
	For a free action of a group $G$ on a finitely separated graph $(E,C)$, the induced action $\alpha$ of $G$ on $\CEC$ factors through an action $\tilde{\alpha}$ of $G$ on the reduced \cstar{}algebra $\CECR$ such that $\P_r(\tilde{\alpha_{g}}(x)) = \alpha_{g}(\P_r(x))$ for all $x \in \CECR$.
\end{lemma}

\begin{proof}
The action $\alpha$ on $C^*(E,C)$ is given on generators by $\alpha_g(P_v) = P_{g \cdot v}$ and $\alpha_g(S_e) = S_{g \cdot e}$ for every $v \in E^0$, $e \in E^1$ and $g \in G$. For each $g \in G$, it is enough to show that the *-homomorphism $\alpha_g:\CEC \to \CEC $ commutes with respect to the conditional expectations, that is, the following diagram commutes: 
	\[
	\begin{tikzcd}
		&\CEC  \arrow{d}{P} \arrow{r}{\alpha_g}  &  \CEC   \arrow{d}{P}  \\
		&C_0(E^0) \arrow{r}{\alpha_g|_{C_0(E^0)}} & C_0(E^0).
	\end{tikzcd}
	\]
	Since $\alpha_g$ is linear and continuous, it is enough to check the commutativity on basic elements of the form $S_{\mu_1} S_{\nu_1}^* \cdots S_{\mu_n} S_{\nu_n}^* \in C(E,C)$ for $\mu\defeq \mu_1\nu_1^*\cdots \mu_n\nu_n^*$ a $C$-separated weakly reduced path. Using \eqref{eq:condexpeconbasiselements} and Lemma \ref{lem:formula-for-P-groupaction}, we have 
	\begin{align*}
		P(\alpha_g(S_{\mu_1} S_{\nu_1}^* \cdots S_{\mu_n} S_{\nu_n}^*)) &= P(S_{g \cdot \mu_1} S_{g \cdot  \nu_1}^* \cdots S_{g \cdot  \mu_n} S_{g \cdot  \nu_n}^*) \\ 
		&= N_{g\cdot \mu}P_{g \cdot  s(\mu)} \\
		&= N_{\mu}P_{g \cdot  s(\mu)} \\
		&= \alpha_g(N_{\mu}P_{s(\mu)}) \\
		&= \alpha_g(P(S_{\mu_1} S_{\nu_1}^* \cdots S_{\mu_n} S_{\nu_n}^*)).
	\end{align*}
	 We conclude that the diagram above in fact commutes. According to Lemma \ref{lemaesperançacondicionalAPR}, $\alpha_g$ factors through an action $\tilde{\alpha_g}: \CECR \to \CECR$ such that $\tilde{\alpha_g}(\Lambda(x)) = \Lambda(\alpha_g(x))$ for all $x \in \CEC$. Finally observe that 
	$$\P_r(\tilde{\alpha_g}(\Lambda(x)))  = \P_r(\Lambda(\alpha_g(x)))	= \P(\alpha_g(x)) = \alpha_g(\P(x)) = \alpha_g(\P_r(\Lambda(x))). $$
This implies that $\P_r \circ \tilde{\alpha_g} = \alpha_g \circ \P_r$, as desired. 
	\end{proof}

\begin{theorem}\label{isomorfismo2}
	With notations as above, there is a canonical isomorphism 
	$$C_r^*(E\times_c G,C \times_c G) \cong C_r^*(E,C) \rtimes_{\delta_c^r} G.$$ 
	Under this isomorphism, the action $\tilde{\gamma}$ on $C_r^*(E\times_c G,C\times_c G)$ induced by the translation action $\gamma$ on $\CEGC$
	corresponds to the dual action $\widehat{\delta_c^r}$ on $C_r^*(E,C)\rtimes_{\delta_c^r}G$.
\end{theorem}

\begin{proof} By Lemmas \ref{lemaesperançacondicionalAPR} and \ref{lemaAprotmesG} and Corollary~\ref{cor:red-separated-cond-exp}, to get the isomorphism from the statement  
it is enough to show that the isomorphism $\phi: \CEGC \to \CEC \rtimes_{\delta_c} G$ from Theorem \ref{isomorgrafosskewseparproducruzado} commutes with the canonical conditional expectations, that is, the following diagram commutes:	
	\[
	\begin{tikzcd}
		&\CEGC  \arrow{d}{Q} \arrow{r}{\phi}  &  \CEC \rtimes_{\delta_c} G   \arrow{d}{P \rtimes G}  \\
		&C_0(E^0 \times G) \arrow{r}{\phi|} & C_0(E^0 ) \rtimes_{\delta_c} G
	\end{tikzcd}
	\]
where $P$ and $Q$ denote the canonical conditional expectation on $C^*(E,C)$ and $C^*(E\times_c G, C\times _c G)$, respectively.
	Set $B=C^*(E,C)$ and $D=C_0(E^0)$. We will adopt the same notation and caveats used in the proof of Theorem \ref{isomorgrafosskewseparproducruzado}, so we consider a $(C\times _cG)$-separated reduced path of the form
	$$(\varsigma, g)= (\mu_1, g)(\nu_1, z_1)^* \cdots (\mu_n,z_{n-1})(\nu_n,z_n)^*.$$
	In particular, $\varsigma = \mu_1\nu_1^*\cdots \mu_n\nu_n^*$ is a $C$-separated reduced path. As seen in the proof of Theorem \ref{isomorgrafosskewseparproducruzado}, we have $\phi(S_{(\varsigma, g)}) = j_G(\chi_{g^{-1}}) j_B(S_{\varsigma})$. Using \eqref{eq:condexpeconbasiselements}, we compute: 
	\begin{align*}
		P \rtimes G(\phi(S_{(\varsigma,g)})) &= P \rtimes G(j_G(\chi_{g^{-1}}) j_B(S_\varsigma ))      \\
		&= j_G(\chi_{g^{-1}}) (j_D\circ P)(S_\varsigma) \\
		&= j_G(\chi_{g^{-1}}) j_D(N_{\varsigma}P_{s(\varsigma)}) \\
        &= \phi(N_{\varsigma }P_{(s(\varsigma),g)}) \\
		&= \phi(N_{(\varsigma, g)}P_{(s(\varsigma),g)}) \\
		&= \phi(Q(S_{(\varsigma,g)})), 
	\end{align*} 
	where here the identity $N_{(\varsigma, g)}= N_{\varsigma}$ is shown by a method similar to the one used in the proof of Lemma \ref{lem:formula-for-P-groupaction}.
This shows the commutativity of the above diagram on a set of linear generators and hence also on the whole algebras by linearity and continuity.

	To finish we need to check the $G$-equivariance of the induced isomorphism 
	$$\phi_r\colon C_r^*(E\times_c G,C \times_c G) \congto C_r^*(E,C) \rtimes_{\delta_c^r} G.$$ 
	But this follows because it factors the $G$-equivariant homomorphism $\phi$ via quotient homomorphisms that are also $G$-equivariant.
	\end{proof}

\begin{corollary}\label{isomorfismo3}
	Let $(E,C)$ be a finitely separated graph and $c:E^1 \to G$ be a labeling function. Then $$C_r^*(E\times_c G,C \times_c G) \rtimes_{\tilde{\gamma},r} G \cong \CECR \otimes \K(\ell^2 G).$$
\end{corollary}

\begin{proof}
	This follows from Corollary \ref{isomorfismo1} and Theorem $\ref{isomorfismo2}$.
	\end{proof}

\begin{corollary}
	For a free action $\theta$ of a group $G$ on a finitely separated graph $(E,C)$, there is a canonical isomorphism
	$$C_r^*(E,C)\rtimes_{\tilde{\theta},r}G\cong C_r^*(E/G,C/G)\otimes\K(\ell^2 G).$$
\end{corollary}

\begin{proof}
	This follows from Corollary \ref{isomorfismo3} and the Gross-Tucker Theorem~\ref{theo:Gross-Tucker-Separated}.	
\end{proof}


\begin{remark}
Unlike what happens with ordinary (unseparated) graphs, in the separated case neither $\gamma$ nor $\tilde{\gamma}$ are amenable in general, meaning that full and reduced crossed products by $\gamma$ or $\tilde\gamma$ might be non-isomorphic.
From our results, we get the following commutative diagram, where except for the vertical map, all others are isomorphisms:
	\[
	\begin{tikzcd}
		\CEGC \rtimes_{\gamma} G \arrow{dr}{\ref{corolariografosskewseparados}} \arrow{rr}{\ref{isomorgrafosskewseparproducruzado}}  & &  \CEC \rtimes_{\delta_c} G \rtimes_{\widehat{\delta}_c} G  \arrow{dl}{\ref{isomordeltamaximal}} \\
		&\CEC \otimes \K(\ell^2 G) \arrow[twoheadrightarrow]{d}  \\
		&\CECR \otimes \K(\ell^2 G)  \\
		\CEGCR \rtimes_{\tilde{\gamma},r} G \arrow{ur}{\ref{isomorfismo3}} \arrow{rr}{\ref{isomorfismo2}}  & &  \CECR \rtimes_{\delta_c^r} G \rtimes_{\widehat{\delta^r_c},r} G  \arrow{ul}{\ref{isomorfismo1}}  \\
	\end{tikzcd}
	\]
\end{remark}

\section{The Fell bundle approach}\label{sec:Fell-bundle}

In this section we analyze the Fell bundles obtained from the spectral decompositions of the coactions $\delta_c$ on the full \cstar{}algebra $C^*(E,C)$ and $\delta_c^r$ on the reduced \cstar{}algebra $C^*_r(E,C)$ of separated graphs $(E,C)$ associated to a labeling function $c\colon E^1\to G$. We show that the Fell bundle from $\delta_c^r$ is a quotient of the Fell bundle from $\delta_c$. This is a proper quotient in general, reflecting the fact that $\delta_c^r$ might not be the normalization of $\delta_c$.

This point of view will highlight the precise connection between the several Fell bundle structures involved in the whole picture and will also make it more explicit where the normalization of $\delta_c$ and the maximalization of $\delta_c^r$ are situated in general. 

We will derive these results as a consequence of a more general result describing the reduced \cstar{}algebra $C^*(\A)_{P,r}$ associated with a certain conditional expectation $P$ onto a \cstar{}subalgebra $B\sbe A_1\sbe C^*(\A)$ for a general Fell bundle $\A=(A_g)_{g\in G}$.  Our result will describe $C^*(\A)_{P,r}$ as the reduced cross-sectional \cstar{}algebra of a quotient Fell bundle of $\A$ by some ideal $\J\sbe \A$ whenever $P$ vanishes on the fibers $A_g$ for $g\not=1$. 

Let us fix a separated graph $(E,C)$ and a labeling function $c: E^1 \to G$ to a group $G$.
We already know from Theorem~\ref{isomordeltamaximal} that the coaction 
$$\delta_c\colon C^*(E,C)\to C^*(E,C)\otimes C^*(G)$$ is maximal. This means that we have a canonical isomorphism 
\begin{equation}\label{eq:iso-spectral-dec-delta_c}
\kappa\colon C^*(\A)\congto C^*(E,C)
\end{equation}
where $\A=(A_g)_{g\in G}$ is the Fell bundle associated to $\delta_c$, that is, 
$$A_g=C^*(E,C)_g:=\{a\in \CEC: \delta_c(a)=a\otimes g\}$$
and $\kappa$ is the $*$-homomorphism on $C^*(\A)$ obtained from the embeddings $A_g\into \CEC$ viewed as a Fell bundle representation. By the way, one could give an alternative proof for the fact that $\delta_c$ is maximal by using the universal property of $\CEC$ in order to prove that $\kappa$ is an isomorphism by producing an inverse homomorphism $\CEC\to C^*(\A)$. This can be done using the dense $*$-subalgebra $L(E,C)\sbe C^*(E,C)$ and the $G$-grading obtained from $c\colon E^1\to G$ as will be described in what follows.

To describe the spectral subspaces $A_g$ inside $C^*(E,C)$, we define 
\begin{multline}\label{eq:spectral-subspaces}
	L(E,C)_g:=\spn\{S_{\mu_1}S_{\nu_1}^* \ldots S_{\mu_n}S_{\nu_n}^* ~|~ \varsigma=\mu_1\nu_1^*\ldots \mu_n\nu_n^* \\ \text{ is a $C$-separated reduced } \text{ path with } c(\varsigma)=g\}.
\end{multline}
It is straightforward to see that this gives $L(E,C)$ a natural algebraic $G$-grading structure meaning that we have a direct sum decomposition $L(E,C) = \displaystyle\oplus_{g\in G}^{\text{alg}} L(E,C)_g$ with grading property: $L(E,C)_g \cdot L(E,C)_h \subseteq L(E,C)_{gh}$ and $L(E,C)_g^* = L(E,C)_{g^{-1}}$ for all $g,h \in G$. Essentially this follows from Proposition \ref{basegrafosseparados}. With this, we get a description of the spectral subspaces of $\CEC$:  

\begin{proposition}\label{propsubespaçosespectrais}
	Given a labeling function $c: E^1 \to G$ on a separated graph $(E,C)$, we have
	\begin{align*}
		C^*(E,C)_g= \overline{L(E,C)}_g
	\end{align*} where $\CEC_g$ is the spectral subspace associated with the coaction $\delta_c$. 
\end{proposition}
\begin{proof}
	It is immediate that $L(E,C)_g$ is contained in $C^*(E,C)_g$ since $$\delta_c(S_{\mu_1}S_{\nu_1}^* \ldots S_{\mu_n}S_{\nu_n}^*) = S_{\mu_1}S_{\nu_1}^* \ldots S_{\mu_n}S_{\nu_n}^* \otimes c(\varsigma) = S_{\mu_1}S_{\nu_1}^* \ldots S_{\mu_n}S_{\nu_n}^* \otimes g$$ where $\varsigma=\mu_1\nu_1^*\ldots \mu_n\nu_n^*$ is a $C$-separated reduced path with $c(\varsigma)=g$. By continuity we have $\overline{L(E,C)}_g \subseteq \CEC_g$. Conversely, suppose that $x \in \CEC$ is such that $\delta_c(x) = x \otimes g$. Since $L(E,C)$ is dense in $\CEC$, $x$ can be approximated by an element $\tilde{x} \in L(E,C)$. Consider the spectral projection $E_g: \CEC \to \CEC_g$ as defined in Section~\ref{sec:coactions}, that is, $E_g = (\id_ \otimes \chi_g)\circ \delta_c$. Since $E_g(x) = x$ we have $$\|x - E_g(\tilde{x})\| =\|E_g(x) - E_g(\tilde{x})\| \leq \|x - \tilde{x}\|.$$ By Proposition \ref{basegrafosseparados} we can consider $\tilde{x}$ as a sum of elements of the form $S_{\mu_1}S_{\nu_1}^* \ldots S_{\mu_n}S_{\nu_n}^*$ with $\varsigma=\mu_1\nu_1^*\ldots \mu_n\nu_n^*$ a $C$-separated reduced path. For each sum factor observe that 
	\begin{align*}
		E_g(S_{\mu_1}S_{\nu_1}^* \ldots S_{\mu_n}S_{\nu_n}^*) &= (\id_ \otimes \chi_g)\circ \delta_c(S_{\mu_1}S_{\nu_1}^* \ldots S_{\mu_n}S_{\nu_n}^*) \\ &= (\id_ \otimes \chi_g)(S_{\mu_1}S_{\nu_1}^* \ldots S_{\mu_n}S_{\nu_n}^* \otimes c(\varsigma)) \\ &= \begin{cases}
		S_{\mu_1}S_{\nu_1}^* \ldots S_{\mu_n}S_{\nu_n}^* & \text{if } c(\varsigma )= g \\
		0 & \text{if } c(\varsigma )\ne  g .
		\end{cases}
	\end{align*}By linearity we have $E_g(\tilde{x}) \in L(E,C)_g$. Therefore, every element $x$ of $\CEC_g$ can be approximated by $E_g(\tilde{x}) \in L(E,C)_g$ as desired. This completes the proof. 
\end{proof}

A similar result also holds for $C^*_r(E,C)$ in place of $C^*(E,C)$ in the above proposition. Indeed, the same proof goes through if we take any \cstar{}algebra completion of $L(E,C)$ for which there exists a coaction $\delta_c^A\colon A\to A\otimes C^*(G)$ satisfying
$$\delta_c^A(S_e)=S_e\otimes c(e)\quad\mbox{and}\quad\delta_c^A(P_v)=P_v\otimes 1$$
for all edges $e\in E^1$ and vertices $v\in E^0$. In this situation we then get, as above, that the associated spectral subspaces are given by
$$A_g:=\{a\in A: \delta_c^A(a)=a\otimes g\}=\overline{L(E,C)}_g^A$$
where we use the superscript $A$ in the closure above to emphasize that the norm of $A$ is being used. This makes a big difference here: although all spectral subspaces `look similar', they are quite different in general due to the possible difference of the norms involved. This difference will be made more explicit in what follows, where we analyze the general situation of a Fell bundle $\A=(A_g)_{g\in G}$.

Let $A$ be a $G$-graded \cstar{}algebra, that is, $A=\overline\oplus_{g\in G}A_g$ for a family of linearly independent closed subspaces $A_g\sbe A$ satisfying $A_g A_h\sbe A_{gh}$ and $A_g^{*}= A_{g^{-1}}$ for all $g,h\in G$. In this situation the family $\A=(A_g)_{g\in G}$ becomes a Fell bundle with the operations inherited from $A$. 

If $J\sbe A$ is a (closed, two-sided) ideal of $A$, then defining $J_g:=J\cap A_g\sbe A_g$, the family $\J=(J_g)_{g\in G}$ becomes an ideal of the Fell bundle $\A$ in the sense of \cite{Exel:Exact_groups_Fell}*{Definition~2.1}, that is, $J_g A_h\sbe J_{gh}$ and $A_g J_h\sbe J_{gh}$ for all $g,h\in G$. We shall say that $\J$ is the ideal of $\A$ induced by $J\sbe A$. As observed in \cite{Exel:Exact_groups_Fell}, $\J=(J_g)_{g\in G}$ is a Fell bundle in its own, and so is also the quotient $\A/\J=(A_g/J_g)_{g\in G}$. Observe, in particular, that $J_1=J\cap A_1$ is an ideal of $A_1$ and we have $J_g=A_g\cdot J_1=J_1\cdot A_g$. Although we do not need this, we also observe that this means that $J_1$ is an $\A$-invariant ideal of $A_1$ as defined in \cite{Kwasniewski-Meyer:Aperiodicity}.

\begin{proposition}\label{prop:Fell-bundle-quotient-conditional}
Let $A=\overline\oplus_{g\in G} A_g$ be a topologically $G$-graded \cstar{}algebra with conditional expectation $E_1\colon A\to A_1$.  
Let $P\colon A\to B$ be a symmetric conditional expectation onto a \cstar{}subalgebra $B\sbe A_1$ with nucleus 
$$\mathcal{N}_P=\{x \in A : P(x^*x)=0\}=\{x\in A : P(xx^*)=0\}\idealin A.$$ 
Assume that $P$ vanishes on $A_g$ for all $g \neq 1$. If $\J=(J_g:=\NN_P\cap A_g)_{g\in G}$ is the ideal of $\A=(A_g)_{g\in G}$ induced by $\NN_P$, then
the quotient maps $A_g\to A_g/J_g$ induce an isomorphism
$$A_{P,r}\cong C_r^*(\A / \J).$$ 
\end{proposition} 
\begin{proof}
Let $P_1:=P|_{A_1}$ denote the restriction of the conditional expectation $P$ to $A_1$. Then $J_1= \mathcal{N}_P \cap A_1=\NN_{P_1}$. 

Since $A$ is topologically graded, it `lies' between $C^*(\A)$ and $C^*_r(\A)$, see \cite{Exel:Partial_dynamical}*{Theorem~19.5}.
More precisely, we have a quotient homomorphism $Q\colon A\to C^*_r(\A)$ which is the `identity' on the fibers $A_g$ viewed as subspaces of both $A$ and $C^*_r(\A)$.

	First of all, note that $\tilde{P_1}: A_1/J_1 \to B$ defined by $\tilde{P_1}(q_1(a)) = P_1(a)$ for $a\in A_1$ is a well-defined faithful conditional expectation because $J_1=\NN_{P_1}$ so that $A_1/J_1=(A_1)_{P_1,r}$. On the other hand, we also have the canonical faithful conditional expectation $\tilde E_1 :=E_1^{\A/\J}\colon C^*_r(\A/\J)\to A_1/J_1$ which vanishes on all fibers $\tilde A_g:=A_g/J_g$ of the quotient Fell bundle $\tilde\A=\A/\J$ for $g\not=1$ and acts as the identity on $\tilde A_1=A_1/J_1$. Therefore $\tilde P_1\circ \tilde E_1$ is also a faithful conditional expectation $C^*_r (\tilde{\mathcal A})\to B$.
	
	We view $C^*_r(\tilde\A)$ as a quotient of $A$ via $\tilde q:=q \circ Q\colon A\onto C^*_r(\A)\onto C^*_r(\tilde\A)$. To finish the proof of the proposition we observe that
	$$\tilde P_1\circ\tilde E_1\circ \tilde q=P.$$
Indeed, to prove this equation it is enough to check it on elements of the fibers $a\in A_g$ for $g\in G$ as these generate $A$ as a Banach space. Now, if $g\not=1$ then both sides of the above equation vanish by the assumption on $P$. And if $g=1$, then $a\in A_1$ and both sides also agree by the construction of $\tilde P_1$. 
The desired result now follows from Corollary~\ref{cor:prop-KM-red-quotient-exp}.	
\end{proof}

Observe that we can apply the above proposition to any \cstar{}algebra $A$ carrying a coaction of $G$, in particular for $A=C^*(\A)$ the full sectional \cstar{}algebra of a Fell bundle $\A$ and a conditional expectation $P\colon C^*(\A)\onto B$ onto a \cstar{}algebra $B\sbe C^*(\A)$ contained in $A_1$ and such that $P$ vanishes on all fibers $A_g$ for $g\not=1$. In this setting we then get a canonical isomorphism
$$C^*(\A)_{P,r}\cong C^*_r(\A/\J).$$

\begin{theorem}\label{theo:spectral-dec-reduced}
	Let $(E,C)$ be a finitely separated graph and let $c: E^1 \to G$ be a labeling function. Then 
	$$C_r^*(E,C) \cong C_r^*(\A/\J),$$
	where $\A$ is the Fell bundle associated to the spectral decomposition of the coaction $\delta_c$ of $G$ on $C^*(E,C)$ induced from $c$ as in Theorem~\ref{isomordeltamaximal}. 
	Moreover, under the above isomorphism, the coaction $\delta_c^r$ on $\CECR$ corresponds to the dual coaction $\delta_{\A/\J}^r$ on $C^*_r(\A/\J)$. Thus $\A/\J$ realizes the spectral decomposition of $\delta_c^\red$.
\end{theorem}
\begin{proof}
Consider $P:\CEC \to C_0(E^0)$ the canonical conditional expectation and the ideal $\mathcal{N}_P = \{x \in \CEC ~|~ P(x^*x)=0\}$ of $\CEC$. Notice that Lemma~\ref{lem:formula-for-P} implies that $P$ vanishes on the fibers $A_g$ for $g\not=1$. With the construction above, $\J = \{J_g\}_{g \in G}$ is the ideal of the Fell bundle $\A$ associated to the spectral subspaces. Remember that $P$ is symmetric and vanishes on $A_g$ for $g \neq 1$ and since $\delta_c$ is a maximal coaction we have $C^*(\A) \cong C^*(E,C)$. Therefore $$C_r^*(E,C) \cong C^*(E,C)_{P,r} \cong C^*(\A)_{P,r} \cong C_r^*(\A/\J)$$ as desired. Since this isomorphism is induced by the quotient map $\A\to \A/\J$, which is a morphism of Fell bundles, it automatically preserves the dual coactions.
\end{proof}

Observe that the above theorem gives an alternative proof for the fact that the coaction $\delta_c^r$ of $G$ on $C^*_r(E,C)$ is normal.

Our results give Fell bundle structures (in particular a topological grading) for the full and reduced \cstar{}algebras of separated graphs for every choice of a labeling function $c\colon E^1\to G$, in particular for the free label $\fl\colon E^1\to \Free$ to the free group generated by its edges, which is an intrinsic label that only depends on the graph. We summarize these results in the following:

\begin{corollary}
For every separated graph $(E,C)$, there exists a Fell bundle $\A$ over the free group $\Free$ on $E^1$ with fibers $A_g=\overline{L(E,C)}_g\sbe C^*(E,C)$ given by 
\begin{multline}\label{eq:spectral-subspaces-free}
	L(E,C)_g:=\spn\{S_{\mu_1}S_{\nu_1}^* \ldots S_{\mu_n}S_{\nu_n}^* ~|~ \varsigma=\mu_1\nu_1^*\ldots \mu_n\nu_n^* \\ \text{ is a $C$-separated reduced } \text{ path with } \fl(\varsigma)=g\}.
\end{multline}
in such a way that the inclusion maps $A_g\into C^*(E,C)$ induce an equivariant  isomorphism $C^*(E,C)\cong C^*(\A)$ with respect to the coactions $\delta_\fl$ and $\delta_\A$.

Moreover, if $(E,C)$ is finitely separated, then the quotient Fell bundle $\A/\J$ by the ideal $\J$ induced by the nucleus $\NN_P$ of the canonical conditional expectation $P\colon C^*(E,C)\to C_0(E^0)$ yields a reduced Fell bundle structure for $C^*_r(E,C)$ in the sense that there is an isomorphism $C^*_r(E,C)\cong C^*_r(\A/\J)$ respecting the coactions $\delta_\fl^r$ and $\delta_{\A/\J}$.
\end{corollary}

Notice that all the components $L(E,C)_g$ intercept $\NN_P$ trivially as $P$ is faithful on $L(E,C)$. This means that the $g$-fiber of $\A/\J$ can also be viewed as a completion of $L(E,C)_g$, but with a different norm than that of $\A_g$, namely, the reduced norm coming from the embedding $L(E,C)_g\into C^*_r(E,C)$. We are going to see in examples that $\J$ is non-trivial in general, even for the free label. We will see that the completions of $L(E,C)_g$ in $C^*(E,C)$ or $C^*_r(E,C)$ are quite different in general, and a precise general description of these seems challenging!

\begin{remark}	
Recall that the normalization of the coaction $\delta_\A$ on the sectional \cstar{}algebra $C^*(\A)$ is always maximal and the coaction $\delta_\A^r$ on $C^*_r(\A)$ is always normal. Moreover, $\delta_\A^r$ is the normalization of $\delta_\A$ while $\delta_\A$ is the maximalization of $\delta_\A^r$.

Now consider a labeling function $c\colon E^1\to G$ on a (finitely) separated graph and the corresponding coactions $\delta_c$ on $C^*(E,C)$ and $\delta_c^r$ on $C^*_r(E,C)$.  Let $\A$ be the Fell bundle associated to $\delta_c$ giving $C^*(E,C)\cong C^*(\A)$ and the corresponding quotient $\A/\J$ by the ideal induced by $\NN_P$ as in Theorem~\ref{theo:spectral-dec-reduced} that gives $C^*_r(E,C)\cong C^*_r(\A/\J)$, where both isomorphisms respect the underlying coactions of $G$.

Then the normalization of $\delta_c$ takes place on an ``exotic'' \cstar{}algebra $C^*_\nu(E,C)\cong C^*_r(\A)$ that can be viewed as a completion of $L(E,C)$ lying between $C^*(E,C)$ and $C^*_r(E,C)$ as illustrated in the following diagram 
 \[
	\begin{tikzcd}
		C^*(\A) \arrow[twoheadrightarrow]{rr} \arrow[d]{}{\cong}  & &C_r^*(\A)\arrow[d]{}{\cong} \arrow[rlll, bend right,swap]{}{\text{ normalization }} \arrow[twoheadrightarrow]{r} & C_r^*(\A/\J) \arrow[d]{}{\cong} \\
		\CEC \arrow[dashrightarrow]{rr} & & C_{\nu}^*(E,C) \arrow[rlll, bend left]{}{\text{ normalization }} \arrow[dashrightarrow]{r} & \CECR . \\
	\end{tikzcd}
\]
	
Similarly, the maximalization of $\delta_c^r$ takes place on an exotic completion $C^*_\mu(E,C)\cong C^*(\A/\J)$ of $L(E,C)$ lying between $C^*(E,C)$ and $C^*_r(E,C)$:
	\[
	\begin{tikzcd}
		C^*(\A) \arrow[twoheadrightarrow]{r} \arrow[d]{}{\cong}  & C^*(\A/\J)\arrow[d]{}{\cong} \arrow[lrrr, bend left]{}{\text{ maximalization }} \arrow[twoheadrightarrow]{rr} & & C_r^*(\A/\J) \arrow[d]{}{\cong} \\
		\CEC \arrow[dashrightarrow]{r}  & C_{\mu}^*(E,C) \arrow[lrrr, bend right,swap]{}{\text{ maximalization }} \arrow[dashrightarrow]{rr} & & \CECR . \\
	\end{tikzcd}
	\]
\end{remark}

\subsection{Some examples}

The following examples intend to illustrate some of our results on the description of \cstar{}algebras of separated graphs and their decompositions in terms of Fell bundles associated with the free label. It should be clear how intricate and complex the structure of those Fell bundles might be; in particular it seems to be difficult to describe precisely the unit fiber of those Fell bundles in certain concrete examples.

\begin{example}\label{ex:ord-graphs}
	Let us revisit the ordinary graph case, that is, the case of a graph $E$ with the trivial separation $C_v=\{ s^{-1}(v)\}$ for all $v\in E^0$.
	We assume that $E$ is row-finite, that is, $s^{-1}(v)$ is finite for every $v$. Then $C^*(E,C)=C_r^*(E,C)=C^*(E)$ is the ordinary graph \cstar{}algebra of $E$, and this equals the closed linear span of elements of the form $S_\mu S_{\nu}^*$ for $\mu, \nu$ (finite) paths on $E$ with $r(\mu)=r(\nu)$. Let us consider the free label $\fl\colon E^1\to \Free$ on $E$ and let $\A=(A_g)_{g\in \Free}$ be the corresponding Fell bundle. We know that
	$$A_g=\cspn\{S_\mu S_\nu^*:\mu,\nu\mbox{ paths on $E$ with $r(\mu)=r(\nu)$ and  $\fl(\mu\nu^{*})=g$}\}.$$
	Viewing paths on $E$ as elements of $\Free$ we have $\fl(\mu\nu^*)=\mu\nu^{-1}$. It follows that the only non-zero fibers $A_g$ are those where $g\in \Free$ is of the reduced form $g=\mu\nu^{-1}$ for $\mu,\nu$ paths on $E$ with $r(\mu)=r(\nu)$. And if $g=\mu\nu^{-1}$ has this form, we have
	$$A_g=\cspn\{S_\mu S_\gamma S_\gamma^* S_\nu^*: \gamma \mbox{ a path on } E\}=S_\mu A_1 S_\nu^*,$$
	where $A_1=\cspn\{S_\gamma S_\gamma^*: \gamma \mbox{ a path on } E\}$ is the unit fiber of $\A$, which is also the usual `diagonal' of $C^*(E)$. It is a commutative \cstar{}algebra whose spectrum $\Omega_E$ is a totally disconnected locally compact space which in case $E$ has no sinks and no sources is homeomorphic to the infinite path space $E^\infty$ of $E$. Indeed, this is related to the well-known fact that $C^*(E)$ is isomorphic to the crossed product $\contz(E^\infty)\rtimes_\theta\Free$ for a suitable partial action $\theta$ of $\Free$ on $E^\infty$, see \cite{Exel:Partial_dynamical}*{Proposition~37.9}. And the Fell bundle of this partial action is exactly the Fell bundle $\A$ above. If $g=\mu\nu^{-1}$ is as above, then $D_g:= \cspn A_g A_g^*=S_\mu A_1 S_\mu^*$,
	and the partial action $\theta$ is induced from the isomorphism $D_{g^{-1}}=S_\nu A_1 S_\nu^*\congto D_g=S_\mu A_1 S_\mu^*$ sending $S_\nu a S_\nu^*\mapsto S_\mu a S_\mu^*=S_\mu S_\nu^*(S_\nu a S_\nu^*)(S_\mu S_\nu^*)^*$ for $a\in A_1$. On the level of $E^\infty$ this corresponds to the homeomorphism $\nu E^\infty\to \mu E^\infty$, $\nu x\mapsto \mu x$. The fact that $C^*(E)$ is nuclear is equivalent (by \cite{BEW-Amenable}*{Corollary 4.11}) to the fact that $\A$ has the approximation property (see \cite{Exel:Partial_dynamical}*{Theorem~36.20}); in particular $C^*(\A)=C^*_r(\A)$, which is compatible with $C^*(E,C)=C^*_r(E,C)=C^*(E)$, of course.
\end{example}

\begin{example}
	Consider the Cuntz separated graph $(A_n,D)$ as in Example~\ref{excuntzgrafoseparado}. As noted in Example~\ref{excanonicografoseparado}, we have $C^*(A_n,D)=C^*(\Free)$, the full \cstar{}algebra of the free group $\Free=\Free_{A_n^1}$;
	and as noted in Examples~\ref{ex:red-Cuntz}, the conditional expectation $P$ coincides with the canonical trace on $C^*(\Free)$ and we have $C^*_r(A_n,D)=C^*_r(\Free)$.  The Fell bundle $\A$ associated to the free label is the trivial Fell bundle $\A=\C\times\Free$ over $\Free$ with all fibers $A_g\cong\C$. In this case we have $\NN_{\tau} \neq 0$ since $\mathbb{F}_n$ is not amenable but $\NN_{\tau} \cap A_g = 0$ for all $g$, so that $\J=0$.
\end{example}

\begin{example}
	Now we consider the separated graph $(E,C)$ with one vertex $E^0= \{v\}$ and four edges $E^1= \{a_1,a_2,b_1,b_2\}$ with separation $C= \{X,Y\}$, $X= \{a_1,a_2\}$, $Y=\{ b_1,b_2\}$; this is a particular case of Example~\ref{exgrafosepararo2}. Then $C^*(E,C)= \mathcal O _2 \star_\C \mathcal O_2$ is the universal unital \cstar{}algebra generated by four isometries $S_1,S_2$, $T_1,T_2$ with $S_1S_1^*+S_2S_2^*=1$ and $T_1T_1^*+T_2T_2^*=1$. And $C^*_r(E,C)= \mathcal O _2\star_\C^r\mathcal O _2$ is the corresponding reduced free product.  
	We first observe that $C^*(E,C)\ne C^*_r(E,C)$ because $C^*_r(E,C)$ is simple by \cite[Proposition 4.2]{Ara-Goodearl:C-algebras_separated_graphs}, but $C^*(E,C)$ is not simple; for instance, the quotient map $\mathcal O_ 2\star_\C \mathcal O _2 \onto \mathcal O_2$ that identifies $S_i \equiv T_i$, $i=1,2$, has non-trivial kernel. 
	
	Hence the canonical conditional expectation $P\colon C^*(E,C)\to C_0(E^0)$ is not faithful; we now show that it remains non-faithful when restricted to the unit fiber $A_1$ of the Fell bundle $\A=(A_g)_{g\in \Free}$ associated to the free label $\fl\colon E^1\to \Free=\Free_4$. 
	
	Consider the unit fiber $C_1$ of the Fell bundle associated to the free label on $C^*(E_X)\cong \Cuntz_2$ as described in Example~\ref{ex:ord-graphs}. As noted there, $C_1$ is generated by the projections $S_\mu S_\mu^*$ with $\mu$ a finite path in the graph $E_X$; it is commutative and isomorphic to $C(E_X^\infty)$ and agrees with the usual `diagonal' of the Cuntz algebra. Similarly we consider the unit fiber (diagonal) $C_2\cong C(E_Y^\infty)$ in $C^*(E_Y)\cong \Cuntz_2$. Of course, $C_1\cong C_2$.
	
	We also consider the \cstar{}subalgebra $B_1\sbe C_1$, isomorphic to $\C^4$, generated by the orthogonal projections $S_1S_1S_1^* S_1^*$, $S_1S_2S_2^* S_1^*$, $S_2S_1S_1^* S_2^*$, $S_2S_2S_2^* S_2^*$, and the corresponding \cstar{}subalgebra $B_2\sbe C_2$. We have $B_1 \subseteq C_1\subseteq C^*(E_X)$ and $B_2\subseteq C_2\subseteq C^*(E_Y)$ and hence
	$$B_1\star_\C B_2 \subseteq C_1\star_\C C_2 \subseteq C^*(E_X)\star_\C C^*(E_Y)= C^*(E,C)$$
	by \cite[Proposition 2.2]{ADEL}. Note that $C_1\star_\C C_2 \subseteq A_1$, but there are elements in $A_1$ that might not come from $C_1\star_\C C_2$, as for instance $S_1T_1T_1^*S_1^*\in A_1$. On the other hand, we also have
	$$B_1\star_\C^r B_2 \subseteq C_1\star_\C^r C_2 \subseteq C^*(E_X)\star_\C^r C^*(E_Y)= C^*_r(E,C)$$
	by the uniqueness of the reduced free product. 	And by \cite[Theorem 1.1]{PS},
	$$B_1\star_\C^r B_2 \cong \C^4 \star_\C^r \C^4 \cong C^*(\Z_4)\star_\C^r C^*(\Z_4) \cong C^*_r (\Z_4\star \Z_4)$$
	is a non-nuclear simple \cstar{}algebra with a unique tracial state. 
	In particular the map $B_1\star_\C B_2\to B_1\star_\C^r B_2$ cannot be injective. We have a commutative diagram 
	\[
	\begin{tikzcd}
		B_1\star_\C B_2  \arrow[rightarrow]{r} \arrow[twoheadrightarrow]{d}  & A_1 \arrow[twoheadrightarrow]{d} \arrow[rightarrow]{rr} & & \mathcal O _2\star_\C \mathcal O_2 \arrow[twoheadrightarrow]{d} \\
		B_1\star_\C^r B_2  \arrow[rightarrow]{r}  & \Lambda (A_1)  \arrow[rightarrow]{rr} & & \mathcal O _2\star_\C^r \mathcal O_2 
	\end{tikzcd}
	\]
	All the horizontal maps are injective. The map $B_1\star_\C B_2\to B_1\star _{\C}^r B_2$ is not injective. Hence the map $A_1\to \Lambda (A_1)$ is not injective. This implies that the restriction of the canonical conditional expectation on $C^*(E,C)\cong C^*(\A)$ to $A_1$ is not faithful.
\end{example}

\begin{example}
	Let $(E,C)$ be the separated graph described in Figure \ref{fig:partiisometry}.
	\begin{center}{
			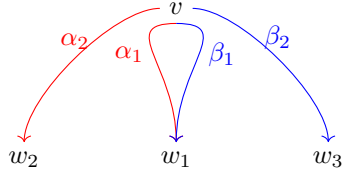
\begin{figure}[htb]
				\begin{tikzpicture}[scale=2]
					\node (v) at (1,1)  {$v$};
					\node (w_1) at (1,0) {$w_1$};
					\node (w_2) at (0,0) {$w_2$};
					\node (w_3) at (2,0) {$w_3$};
					\draw[->,red]  (v.west) ..  node[above]{$\alpha_2$} controls+(left:3mm) and +(up:3mm) ..
					(w_2.north) ;
					\draw[->,red] (v.south) .. node[below, left]{$\alpha_1$}  controls+(left:4mm) and +(up:5mm) ..
					(w_1.north);
					\draw[->,blue] (v.south) .. node[below, right]{$\beta_1$}
					controls+(right:4mm) and +(up:5mm) ..
					(w_1.north);
					\draw[->,blue] (v.east) .. node[above]{$\beta_2$}
					controls+(right:3mm) and +(up:3mm) ..
					(w_3.north);
				\end{tikzpicture}
				\caption{The separated graph of a partial isometry}
				\label{fig:partiisometry}
		\end{figure}}
	\end{center}
	We set $X= \{\alpha_1,\alpha_2\}$ and $Y= \{ \beta_1,\beta_2\}$ so that $C= \{ X,Y\}$. 
	Here the interesting algebra is the full corner algebra $A= vC^*(E,C)v$ of $C^*(E,C)$; to simplify notation, here and in the following, we shall write $v$ for the projection $P_v\in C^*(E,C)$ corresponding to the vertex $v$.    
	By \cite[Lemma 5.5(1)]{Ara:Purely_infinite}, there is an isomorphism between $A$ and the universal \cstar{}algebra generated by a single partial isometry $s$. This isomorphism sends $\beta_1\alpha_1^*$ to $s$. We will write $s= \beta_1\alpha_1^*\in A$. In order to apply the results about the normal form of a linear basis of $L(E,C)$, we have to select edges $e_X\in X$ and $e_Y\in Y$. There is a standard choice in this case, given by taking $e_X=\alpha_2$ and $e_Y = \beta_2$. Then a $C$-separated reduced path $\gamma $ such that $s(\gamma)= r(\gamma)= v$ is a path of the form $$\gamma = s^{n_1}(s^*)^{n_2}\cdots \cdots (s^*)^{n_r},$$ 
	where $n_1,n_r\ge 0$ and $n_i >0$ for $i= 2,\dots , r-1$, such that the path $\gamma $ does not contain any subpath of the form $ss^*s$ or $s^*ss^*$.

	Now we study the free label on the \cstar{}algebra $A$. Let $\mathbb F$ be the free group on $E^1$. The free label induces a topological grading on $A$. Let $g =  \beta_1 \alpha_1^{-1}\in \mathbb F$. Then the only elements $h$ of $\mathbb F$ for which $A_h\ne 0$ are those of the form $g^i$ for $i\in \Z$.
	We thus obtain a topological grading     
	$$A= \ol{ \bigoplus _{i\in \Z} A_{g^i}},$$
	induced by the free label.
	For instance, the algebra $A_1$ contains, among other things, elements of the form $\gamma _ n = s^n(s^*)^n$ and $\gamma_{-n}= (s^*)^ns^n$, for $n\ge 1$.      
\end{example}

In order to better understand the structure of the \cstar{}algebra $C^*(E,C)$ and its full corner $A=vC^*(E,C)v$, we introduce another \cstar{}algebra $C^*(F,D)$ associated to another separated graph $(F,D)$.    

\begin{example}
	\label{exam:lamplighter} Let $(F,D)$ be the separated graph
	described in Figure \ref{fig:lampgroup}, with $D_v=\{ X', Y'\}$ and
	$X'=\{ \alpha _1,\alpha_2\}$ and $Y'=\{ \beta _1,\beta _2 \}$. By
	\cite[Lemma 5.5(2)]{Ara:Purely_infinite}, we have
	$$vC^*(F,D)v\cong C^*( (\star _{\Z}\Z_2)\rtimes \Z) \,$$
	where $\Z$ acts on $\star_{\Z} \Z_2$ by shifting the factors of the
	free product.

	\begin{center}{
			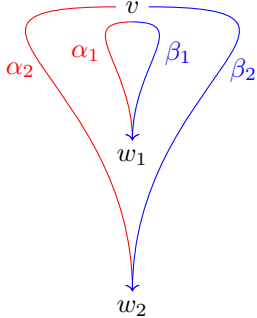
\begin{figure}[htb]
				\begin{tikzpicture}[scale=2]
					\node (v) at (0,1)  {$v$};
					\node (w_1) at (0,0) {$w_1$};
					\node (w_2) at (0,-1) {$w_2$};
					\draw[->,red]  (v.south) ..  node[below, left]{$\alpha_1$} controls+(left:4mm) and +(up:5mm) ..
					(w_1.north) ;
					\draw[->,red] (v.west) .. node[below, left]{$\alpha_2$}  controls+(left:14mm) and +(up:14mm) ..
					(w_2.north);
					\draw[->,blue] (v.east) .. node[below, right]{$\beta_2$}
					controls+(right:14mm) and +(up:14mm) ..
					(w_2.north);
					\draw[->,blue] (v.south) .. node[below, right]{$\beta_1$}
					controls+(right:4mm) and +(up:5mm) ..
					(w_1.north);
				\end{tikzpicture}
				\caption{The separated graph underlying the (free) lamplighter group}
				\label{fig:lampgroup}
		\end{figure}}
	\end{center}
	We also select $e_{X'}:= \alpha _2$ and $e_{Y'}:=\beta_2$ in this case. Write $B=vC^*(F,D)v$. The \cstar{}algebra $B$ is generated by two partial isometries $s:=\beta_1\alpha_1^*$ and $t:= \beta_2\alpha_2^*$. Observe that $z= s+t$ is a unitary in $B$ which corresponds to the unitary which generates the copy of $\Z$ in the isomorphism $B\cong C^*( (\star_{\Z}\Z_2)\rtimes \Z)$. Set $p_0^+=\alpha_1\alpha_1^*$, $p_0^-= 1-p_0^+= \alpha_2\alpha_2^*$, and $p_n^{\pm}= z^np_0^{\pm}(z^*)^n$ for $n\in \Z$. Then $u_n=1-2p_n^+$ is the unitary corresponding to the generator of the $n$-th copy of $\Z_2$ under the isomorphism	$B\cong C^*( (\star _{\Z}\Z_2)\rtimes \Z)$. See \cite{Ara:Purely_infinite}. 
	
	We now pass to the reduced \cstar{}algebras. It is shown in \cite[Proposition 5.6]{Ara:Purely_infinite} that 
	$$vC_r^*(F,D)v \cong C^*_r((\star _\Z \Z_2)\rtimes \Z)\cong (C^*_r(\star _\Z \Z_2) )\rtimes \Z .$$
	In particular, we can see from the proof of \cite[Proposition 5.6]{Ara:Purely_infinite} that the canonical conditional expectation $P\colon C^*(F,D)\to v\C + w_1\C + w_2\C$ restricts to the usual trace $\tau$ on $C^*((\star _\Z \Z_2)\rtimes \Z)$ when restricted to $B$. Let us denote by $B_{\tau,r}$ the reduction of $B$ induced by the tracial state $\tau$ on $B$. Note that $B_{\tau,r} = vC^*(F,D)_{P,r}v$. Observe that we have a natural $\Z$-grading in $B$ induced by the crossed product $B\cong  C^*(\star _\Z \Z_2)\rtimes \Z$. This is indeed induced by the labeling $F^1  \to \langle g \rangle $, where $\langle g\rangle $ is a multiplicative infinite cyclic group, given by $$\alpha_1\mapsto 1, \quad \alpha_2\mapsto 1,\quad \beta_1\mapsto g,\quad \beta_2\mapsto g.$$
	Note that the elements $s= zp_0^+$ and $t=zp_0^-$ have both degree $g$ with respect to the above labeling, on so the element $z= s+t$ belongs to the component $B_g$, as desired. The projections $p_n^+=z^np_0^+(z^*)^n$ belong to $B_1$, and they, together with $1$, generate the component $B_1$. 

Observe that the grading induced by the free label on $C^*(F,D)$ is finer than the above mentioned $\Z$-grading. Indeed the restriction to $B$ of the grading induced by the free label gives a $\mathbb F_2$-grading, with generators $s$ and $t$, and the $\Z$-grading is obtained via the group homomorphism $\mathbb F_2 \to \langle g \rangle $ given by sending both $s$ and $t$ to $g$. 
	
	We consider the \cstar{}algebra $B$, with its structure of topological $\Z$-graded \cstar{}algebra as given above, and the tracial state $\tau$ on $B$ seen as a conditional expectation onto $\C v\subseteq B_1$. Let $\mathcal B$ be the associated Fell bundle, and consider the corresponding ideal $J$ of $B_1$, which is the kernel of the restriction of $\tau$ to $B_1\cong C^*(\star_\Z \Z_2)$, and the Fell bundle ideal $\mathcal J$ of $\mathcal B$ induced by $J$.  We then have, by Proposition~\ref{prop:Fell-bundle-quotient-conditional},
	$$B_{\tau,r} \cong C^*_r(\mathcal B/\mathcal J),$$
	but note that since 
	$$C^*(\mathcal B/\mathcal J) \cong C^*_r(\star_\Z \Z_2) \rtimes \Z$$
	and $C^*_r(\star_\Z \Z_2) \rtimes \Z \cong C^*_r(\star_\Z \Z_2) \rtimes _r\Z $ because $\Z$ is an amenable group, we indeed have
	$$C^*(\mathcal B/\mathcal J)  \cong C^*_r(\star_\Z \Z_2) \rtimes \Z\cong C^*_r(\star_\Z \Z_2) \rtimes_r \Z \cong C^*_r(\mathcal B/\mathcal J) \cong  B_{\tau,r}. $$
	
	Now we consider the two \cstar{}algebras $C^*(E,C)$ and $C^*(F,D)$ above. By 
	\cite[Lemma 5.5(2) and Proposition~5.8]{Ara:Purely_infinite}, there is a trace-preserving unital embedding $vC^*(E,C)v \to vC^*(F,D)v$. That is, in the notation introduced above, we have a trace-preserving unital embedding $A\to B$.
	Note that, also with the notation introduced above, the partial isometry  denoted by $s$ in $A$ maps to the partial isometry denoted also by $s$ in $B$, cf. the proof of \cite[Proposition 5.8]{Ara:Purely_infinite}. Now it is clear that $A_{g^n}= B_{g^n}\cap A$, and in particular the free label on $(E,C)$ induces the same grading on $A$ as the restriction of the $\Z$-grading of $B$. Let $J_0$ be the kernel of the restriction of $\tau$ to $A_1$, and let $\mathcal J_0$ be the corresponding Fell bundle ideal for $\mathcal A$, where $\mathcal A$ is the Fell bundle associated to the topological $\Z$-grading of $A$. Then $J_0= J\cap A_1$ and $\mathcal J_0=\mathcal J \cap \mathcal A$. Thus $\mathcal A /\mathcal J_0$ is a Fell subbundle of $\mathcal B/\mathcal J$. Since $\Z$ is an amenable group, \cite[Proposition 21.7]{Exel:Partial_dynamical} gives that the natural map  
	$$\iota \colon  C^*(\mathcal A/\mathcal J_0) \to C^*(\mathcal B/ \mathcal J)$$
	is injective, and thus the restriction of the trace $\tau$ to $C^*(\mathcal A /\mathcal J_0)$ is faithful, which implies
	$$A_{\tau,r} = C^*_r(\mathcal A/\mathcal J_0)= C^*(\mathcal A/\mathcal J_0).$$
	On the other hand, $A$ is not exact by \cite{BrenkenNiu}, and $A_{\tau,r}=vC_r^*(E,C)v$ is exact by \cite[Remark 3.10]{Ara-Goodearl:C-algebras_separated_graphs}, so it follows that $\mathcal J_0\ne 0$, and so the free label on $C^*(E,C)$ is not faithful on the unit fiber $C^*(E,C)_1$.
	Moreover, since $A_1 \subseteq B_1^{\f}$, where $B_1^{\f}$ is the neutral component of the grading on $B$ induced by the free label, the conditional expectation on $C^*(F,D)$ is also not faithful on the unit fiber $C^*(F,D)_1$.   
	\end{example}

\section*{Acknowledgments}

The authors would like to thank the anonymous referee for his/her careful reading of the paper and his/her very helpful comments, which have improved the exposition of the paper.

\end{document}